\documentclass[12pt]{amsart}
\allowdisplaybreaks
\usepackage[english]{babel}
\usepackage[pdftex,textwidth=400pt,marginratio=1:1]{geometry}
\usepackage{amsfonts}
\usepackage[dvips]{graphics}
\usepackage[colorinlistoftodos]{todonotes}
\usepackage{amsmath}
\usepackage{amsthm}
\usepackage{amssymb}
\usepackage{bbm}
\usepackage{cancel}
\usepackage{color}
\usepackage[curve]{xypic}
\usepackage{graphicx}
\newtheorem{theorem}{Theorem}[section]
\newtheorem{corollary}[theorem]{Corollary}
\newtheorem{proposition}[theorem]{Proposition}
\newtheorem{lemma}[theorem]{Lemma}
\newtheorem{conjecture}[theorem]{Conjecture}

\theoremstyle{definition}
\newtheorem{definition}[theorem]{Definition}

\newtheorem{remark}[theorem]{Remark}

\theoremstyle{property}

\DeclareFontFamily{OT1}{rsfs}{}
\DeclareFontShape{OT1}{rsfs}{n}{it}{<-> rsfs10}{}
\DeclareMathAlphabet{\curly}{OT1}{rsfs}{n}{it}

\newcommand\I{\mathcal I}

\renewcommand\O{\mathcal O}
\newcommand\PP{\mathbb P}

\newcommand\cA{\mathcal A}

\newcommand\E{\mathbb E}
\newcommand\C{\mathbb C}

\newcommand\sfT{\mathsf T}
\newcommand\FF{\mathbb F}

\newcommand\sfZ{\mathsf Z}

\newcommand\Q{\mathbb Q}

\newcommand\Z{\mathbb Z}
\newcommand\cZ{\mathcal Z}
\newcommand\Eu{\mathrm{Eu}}

\newcommand\Res{\mathrm{Res}}
\newcommand\inst{\mathrm{inst}}
\newcommand\mono{\mathrm{mono}}

\newcommand\s{\mathfrak s}

\newcommand\SU{\mathrm{SU}}

\newcommand\vd{\mathrm{vd}}

\newcommand\vir{\mathrm{vir}}

\newcommand\SW{\mathrm{SW}}
\newcommand\VW{\mathrm{VW}}
\newcommand\td{\mathrm{td}}
\newcommand\rk{\operatorname{rk}}

\newcommand\tr{\operatorname{tr}}

\newcommand\ch{\operatorname{ch}}

\renewcommand\hom{\mathcal{H}{\it{om}}}

\newcommand\Hilb{\operatorname{Hilb}}

\newcommand\INTO{\ar@{^{(}->}[r]}

\DeclareRobustCommand{\SkipTocEntry}[4]{}

\begin{document}
\title[Refined $\SU(3)$ Vafa-Witten invariants and modularity]{Refined $\SU(3)$ Vafa-Witten invariants and modularity}
\author[L.~G\"ottsche and M.~Kool]{Lothar G\"ottsche and Martijn Kool}
\maketitle

\vspace{-0.8cm}

\begin{abstract}
We conjecture a formula for the refined $\SU(3)$ Vafa-Witten invariants of any smooth surface $S$ satisfying $H_1(S,\Z) = 0$ and $p_g(S)>0$. The unrefined formula corrects a proposal by Labastida-Lozano and involves unexpected algebraic expressions in modular functions. We prove that our formula satisfies a refined $S$-duality modularity transformation.

We provide evidence for our formula by calculating virtual $\chi_y$-genera of moduli spaces of rank 3 stable sheaves on $S$ in examples using Mochizuki's formula. Further evidence is based on the recent definition of refined $\mathrm{SU}(r)$ Vafa-Witten invariants by Maulik-Thomas and subsequent calculations on nested Hilbert schemes by Thomas (rank 2) and Laarakker (rank 3). 
\end{abstract}
\thispagestyle{empty}

\vspace{-0.2cm}

\section{Introduction} \label{intro}

\subsection{Physics background} 

In 1994, C.~Vafa and E.~Witten proposed tests for $S$-duality of $N=4$ supersymmetric Yang-Mills theory on a real 4-manifold $M$ \cite{VW}. This theory involves coupling constants $\theta, g$ combined as follows
$$
\tau := \frac{\theta}{2 \pi} + \frac{4 \pi i}{g^2}. 
$$
$S$-duality predicts that the transformation $\tau \mapsto -1/\tau$ maps the partition function for gauge group $G$ to the partition function with Langlands dual gauge group $^{L}G$. Vafa-Witten consider $M$ underlying a smooth projective
surface $S$ over $\C$ and $G = \mathrm{SU}(r)$. Furthermore, they consider a topological twist of the original theory. Roughly speaking, the partition function is the generating function of topological Euler characteristics of moduli spaces of
instantons and the transformation property implies that it is a modular form.

Referring in part to the mathematics literature \cite{Kly, Yos1, Yos2, Nak1, Nak2}, Vafa-Witten perform non-trivial modularity checks for $S = \PP^2, K3$, blow-ups, and ALE spaces. Most of these checks are for rank $r=2$. In \cite[Sect.~5]{VW}, they predict a formula for their invariants when $r=2$ and $S$ contains a smooth curve in its canonical linear system. This formula was generalized to arbitrary $S$ satisfying $p_g(S)>0$ by R.~Dijkgraaf, J.-S.~Park, and B. Schroers \cite{DPS}. At the time, there were no purely mathematical verifications of these ``general type'' formulae, due to the lack of a precise algebro-geometric definition of Vafa-Witten invariants.

\subsection{Tanaka-Thomas's definition} Recently, Y.~Tanaka and R.~P.~Thomas proposed an algebro-geometric definition of $\SU(r)$ Vafa-Witten invariants \cite{TT1, TT2}. Let $S$ be a smooth projective surface over $\C$ with polarization $H$. For any line bundle $L$, consider the moduli space of $H$-stable Higgs pairs
$$
N^{\perp} := N_S^H(r,L,c_2) = \big\{ (E,\phi) \, : \, \det E \cong L, \ \tr \phi = 0, \ c_2(E) =c_2 \big\}.
$$
Here $E$ is a rank $r$ torsion free sheaf, $\phi : E \rightarrow E \otimes K_S$ is a morphism, called the Higgs field, and the pair $(E,\phi)$ satisfies a (Gieseker) stability condition with respect to $H$. Assuming stability and semistability coincide, Tanaka-Thomas show that $N^{\perp}$ admits a (delicate) symmetric perfect obstruction theory. The $\C^*$ scaling action on the Higgs field lifts to $N^\perp$, which is therefore non-compact. However, the fixed locus $(N^\perp)^{\C^*}$ is compact and the $\SU(r)$ Vafa-Witten invariants are defined by virtual localization
\begin{equation} \label{defloc}
\VW_S^H(r,L,c_2) := \int_{[N_S^H(r,L,c_2)^{\C^*}]^{\vir}} \frac{1}{e(N^\vir)} \in \Q.
\end{equation}
There are two types of $\C^*$ fixed Higgs pairs $(E,\phi)$. Higgs pairs with $\phi = 0$ form a component isomorphic to the moduli space $M:=M_S^H(r,L,c_2)$ of $H$-stable rank $r$ torsion free sheaves $E$ with $\det E \cong L$ and $c_2(E) = c_2$. We call this the \emph{instanton branch} \cite{DPS}. The contribution of the instanton branch to \eqref{defloc} is
$$
(-1)^{\vd(M)} e^\vir(M) \in \Z,
$$ 
i.e.~the virtual Euler characteristic defined by B.~Fantechi and the first-named author in \cite{FG} (see also \cite{CFK}) and where
\begin{equation} \label{vd}
\vd = 2r c_2 - (r-1) c_1^2 - (r^2-1)\chi(\O_S)
\end{equation}
is the virtual dimension of $M$. We refer to the connected components of $(N^\perp)^{\C^*}$ consisting of Higgs pairs with $\phi \neq 0$ as the \emph{monopole branch} \cite{DPS}. Since Vafa-Witten invariants are defined by localization, the contribution of the monopole branch is in general a rational number. When $\deg K_S < 0$ or $K_S \cong \O_S$, there are no contributions from the monopole branch,
$M$ is smooth, and $e^{\vir}(M)$ equals the topological Euler characteristic $e(M)$ of $M$ \cite[Prop.~7.4]{TT1}.

Tanaka-Thomas's definition has been generalized in two directions:
\begin{itemize}
\item adding strictly $H$-semistable Higgs pairs \cite{TT2}, 
\item $y$-refinement $\VW_S^H(r,L,c_2,y)$ defined by D.~Maulik and Thomas \cite{MT} (see also \cite{Tho}).
\end{itemize}
The contribution of the instanton branch to $\VW_S^H(r,L,c_2,y)$ equals \cite{Tho}
\begin{equation*} \label{MTinstanton}
(-1)^{\vd(M)} \overline{\chi}_{-y}^{\vir}(M) := (-1)^{\vd(M)} y^{-\frac{\vd(M)}{2}} \chi_{-y}^{\vir}(M) \in \Z[y^{\pm \frac{1}{2}}],
\end{equation*}
which is the (normalized) virtual $\chi_y$-genus defined in \cite{FG}.

\subsection{Previous calculations} When $\deg K_S < 0$ or $K_S \cong \O_S$, Vafa-Witten invariants are Euler characteristics of smooth moduli spaces (assuming ``stable equals semistable''). Modularity of generating functions of Euler characteristics of smooth moduli spaces of sheaves has been verified by direct calculation in many examples; mostly for rank 2 (see references in \cite{GK1}). For some higher rank calculations, see \cite{BN, Koo, Man, Moz, Wei}.

Henceforth, $S$ is a smooth projective surface such that $b_1(S) = 0$ and $p_g(S) >0$. Except for $S = K3$ or an elliptic surface \cite{GH, Yos3}, until recently very few calculations of invariants of moduli spaces $M_S^H(r,L,c_2)$ were known. The following invariants were recently calculated for some examples of general type surfaces:
\begin{itemize}
\item $\chi_{y}^{\vir}(M_S^H(2,c_1,c_2))$ for roughly $c_2 \leq 7$, and $c_2 \leq 30$ when $y=-1$ \cite{GK1},\footnote{See \cite{GK2} for refinements to virtual elliptic genus / cobordism class.}
\item monopole contribution to $\VW_S^H(2,L,c_2)$ for $c_2 \leq 3$ in \cite{TT1}.
\end{itemize}
In fact, conjectural formulae exist in both cases following from (generalizations of) a formula from Vafa-Witten \cite[(5.38)]{VW}. See Remark \ref{rank2conj}. The above-mentioned calculations all match the conjectural formulae, which provides strong evidence that Tanaka-Thomas's definition is correct.

\subsection{Instanton branch}

Our first conjecture concerns the virtual $\chi_y$-genera of $M_S^H(3,c_1,c_2)$ when there are no rank 3 strictly Gieseker $H$-semistable sheaves on $S$ with Chern classes $c_1,c_2$. The rank 1 case was covered by \cite{GS} and, as just mentioned, a conjectural formula for the rank 2 case was proposed in \cite{GK1}.

Denote by $\SW(a)$ the Seiberg-Witten invariant of $S$ for class $a \in H^2(S,\Z)$.\footnote{We use Mochizuki's convention: $\SW(a) = \widetilde{\SW}(2a-K_S)$ with $\widetilde{\SW}(b)$ the usual Seiberg-Witten invariant in class $b$.} We refer to $a$ as a Seiberg-Witten basic class when $\SW(a) \neq 0$. 

The $A_2$-lattice consists of $\Z^2$ together with bilinear form $\langle v , w \rangle := v^t A w$ defined by
\begin{equation} \label{matrixA}
A:= \left(\begin{array}{cc} 2 & -1 \\ -1 & 2 \end{array}\right).
\end{equation}
We also need the dual lattice $A_2^\vee$ consisting of $\Z^2$ and  $\langle v , w \rangle^\vee := v^t A^\vee w$ where
\begin{equation} \label{matrixAinv}
A^\vee = A^{-1} =  \frac{1}{3} \left(\begin{array}{cc} 2 & 1 \\ 1 & 2 \end{array}\right).
\end{equation}
Let $\epsilon := e^{\frac{2 \pi i}{3}}$. The following theta functions feature in our conjectures
\begin{align*}
\Theta_{A_2,(0,0)}(x,y) :=& \, \sum_{v \in \Z^2} (x^2)^{\frac{1}{2} \langle v,v\rangle} e^{2 \pi i \langle v, (z,z) \rangle} \\
=&\, \sum_{(m,n) \in \Z^2} x^{2(m^2 -mn +n^2)} y^{m+n}, \\
\Theta_{A_2,(1,0)}(x,y) :=& \, \sum_{v \in \Z^2} (x^2)^{\frac{1}{2}\langle v+(\frac{1}{3},-\frac{1}{3}),v +(\frac{1}{3},-\frac{1}{3}) \rangle} e^{2 \pi i \langle v, (z,z) \rangle} \\
=& \, \sum_{(m,n) \in \Z^2} x^{2(m^2 -mn +n^2+m - n +\frac{1}{3})} y^{m+n}, \\
\Theta_{A_2^\vee,(0,0)}(x,y) :=& \, \sum_{v \in \Z^2} (x^6)^{\frac{1}{2} \langle v,v\rangle^\vee} e^{2 \pi i \langle v, (z,z) \rangle^\vee} \\
=& \, \sum_{(m,n) \in \Z^2} x^{2(m^2 +mn +n^2)} y^{m+n},  \\
\Theta_{A_2^\vee,(0,1)}(x,y) :=& \, \sum_{v \in \Z^2} (x^6)^{\frac{1}{2} \langle v,v\rangle^\vee} e^{2 \pi i \langle v, (z,z) + (1,-1) \rangle^\vee} \\
=& \,\sum_{(m,n) \in \Z^2} \epsilon^{m-n} x^{2(m^2 +mn +n^2)} y^{m+n},
\end{align*}
where $x = e^{\frac{\pi i \tau}{3}}$, $y = e^{2 \pi i z}$, $(\tau,z) \in \mathfrak{H} \times \C$ are the modular parameters (which feature later), and $\mathfrak{H}$ denotes the upper half plane. We also use the normalized Dedekind eta function $\overline{\eta}(x) = \prod_{n>0} (1-x^n)$. Furthermore, we abbreviate $\chi := \chi(\O_S)$, $K := K_S$, $b_i(S):=b_i$, $p_g:=p_g(S)$, and $e:=e(S) = \int_S c_2(S)$.
\begin{conjecture} \label{conj1}
Let $S$ be a smooth projective surface satisfying $b_1 = 0$ and $p_g>0$. Let $H, c_1, c_2$ be chosen such that there are no rank 3 strictly Gieseker $H$-semistable sheaves on $S$ with Chern classes $c_1,c_2$ and let $M:=M_S^H(3,c_1,c_2)$. Then $\overline{\chi}_{-y}^{\vir}(M)$ equals the coefficient of $x^{\vd(M)}$ of 
\begin{align*}
&9 \Bigg( \frac{1}{3 \prod_{n=1}^{\infty} (1-x^{2n})^{10}(1-x^{2n}y)(1-x^{2n} y^{-1})} \Bigg)^{\chi} \Bigg( \frac{\Theta_{A^\vee_2, (0,1)}(x,y)}{3 \overline{\eta}(x^6)^3} \Bigg)^{-K^2} \\
&\qquad \times \sum_{a,b} \SW(a) \, \SW(b) \, \epsilon^{(a-b)c_1} \, Z_{+}(x,y)^{ab} \, Z_{-}(x,y)^{(K-a)(K-b)}, 
\end{align*}
where the sum is over all $(a,b) \in H^2(S,\Z) \times H^2(S,\Z)$ and $Z_{\pm}(x,y)$ are the solutions to the following quadratic equation in $\zeta$
\begin{align*}
\zeta^2 - (Z(x,y)^2+3Z(x,y)Z(x,1)) \, \zeta +Z(x,y)+3Z(x,1) = 0,
\end{align*}
where $Z(x,y) := \frac{\Theta_{A^\vee_2, (0,0)}(x,y)}{\Theta_{A^\vee_2, (0,1)}(x,y)}$.
\end{conjecture}

Let $S,H,c_1$ be chosen such that there exist no rank 3 strictly Gieseker $H$-semistable sheaves on $S$ with first Chern class $c_1$. We define 
\begin{equation} \label{instantongenfun}
\sfZ_{S,H,r,c_1}^{\inst}(q,y) := \, q^{-\frac{1}{2r}\chi + \frac{r}{24} K^2} \sum_{c_2} \overline{\chi}_{-y}^{\vir}(M_S^H(r,c_1,c_2)) \, q^{\frac{\vd}{2r}},
\end{equation}
which we refer to as the instanton contribution to the Vafa-Witten generating function and where $\vd$ is given by \eqref{vd}. The normalization becomes important in Section \ref{modintro} when we discuss modularity. Let $\phi_{-2,1}$ be the weak Jacobi form of weight $-2$ and index $1$ with Fourier expansion 
$$
\phi_{-2,1}(q,y) = (y^{\frac{1}{2}} - y^{-\frac{1}{2}})^2 \prod_{n=1}^{\infty} \frac{(1-q^ny)^2(1-q^n y^{-1})^2}{(1-q^n)^4}.
$$
Denote the discriminant modular form by $\Delta(q) = q \prod_{n>0} (1-q^n)^{24}$.
\begin{corollary} \label{corconj1}
Let $S$ be a smooth projective surface satisfying $b_1 = 0$ and $p_g>0$. Let $H, c_1$ be such that there are no rank 3 strictly Gieseker $H$-semistable sheaves on $S$ with first Chern class $c_1$. Assume Conjecture \ref{conj1} holds for $S,H,c_1$ and all $c_2$. Then we have
{\scriptsize{
\begin{align*}
\frac{\sfZ_{S,H,3,c_1}^{\inst}(q,y)}{(y^{\frac{1}{2}} - y^{-\frac{1}{2}})^{\chi}} =\, & 3\Bigg( \frac{1}{3\phi_{-2,1}(q^{\frac{1}{3}},y)^{\frac{1}{2}} \Delta(q^{\frac{1}{3}})^{\frac{1}{2}} } \Bigg)^{\chi} \Bigg(\frac{\Theta_{A_2^\vee,(0,1)}(q^{\frac{1}{6}},y)}{3\eta(q)^3}  \Bigg)^{-K^2} \\
&\qquad\qquad \times \sum_{a,b} \SW(a) \, \SW(b) \,\epsilon^{(a-b) c_1} \, Z_+(q^{\frac{1}{6}},y)^{ab} \, Z_-(q^{\frac{1}{6}},y)^{(K-a)(K-b)} \\
&+3\epsilon^{2c_1^2} \Bigg( \frac{1}{3\phi_{-2,1}(\epsilon^2 q^{\frac{1}{3}},y)^{\frac{1}{2}}  \Delta(\epsilon^2 q^{\frac{1}{3}})^{\frac{1}{2}} } \Bigg)^{\chi} \Bigg(\frac{\Theta_{A_2^\vee,(0,1)}(\epsilon q^{\frac{1}{6}},y)}{3\eta(q)^3}  \Bigg)^{-K^2} \\
&\qquad\qquad \times \sum_{a,b} \SW(a) \, \SW(b) \,\epsilon^{(a-b) c_1} \, Z_+(\epsilon q^{\frac{1}{6}},y)^{ab} \, Z_-(\epsilon q^{\frac{1}{6}},y)^{(K-a)(K-b)}  \\
&+3(-1)^\chi \epsilon^{c_1^2} \Bigg( \frac{1}{3\phi_{-2,1}(\epsilon q^{\frac{1}{3}},y)^{\frac{1}{2}} \Delta(\epsilon q^{\frac{1}{3}})^{\frac{1}{2}} } \Bigg)^{\chi} \Bigg(\frac{\Theta_{A_2^\vee,(0,1)}(\epsilon^2 q^{\frac{1}{6}},y)}{3\eta(q)^3}  \Bigg)^{-K^2} \\
&\qquad\qquad \times \sum_{a,b} \SW(a) \, \SW(b) \,\epsilon^{(a-b) c_1} \, Z_+(\epsilon^2 q^{\frac{1}{6}},y)^{ab} \, Z_-(\epsilon^2 q^{\frac{1}{6}},y)^{(K-a)(K-b)}. 
\end{align*}
}}
\end{corollary}
\begin{proof}
Define 
$$
\sfZ_{S,H,3,c_1}^{\inst}(x,y) := \sum_{c_2} \overline{\chi}_{-y}^{\vir}(M_S^H(3,c_1,c_2)) \, x^{\vd}.
$$
Denote the formula of Conjecture \ref{conj1} by $\psi_{S,c_1}(x,y) = \sum_{n \geq 0} \psi_n(y) \, x^n$. Then 
$$
\sfZ_{S,H,3,c_1}^{\inst}(x,y) = \sum_{n \equiv -2c_1^2 - 8\chi \mod 3} \psi_n(y) \, x^{n} = \sum_{k=0}^{2} \frac{1}{3}\epsilon^{k(2c_1^2+ 8\chi)} \psi_{S,c_1}(\epsilon^k x,y).
$$
The result follows by setting $x = q^{\frac{1}{6}}$ and multiplying by $q^{-\frac{1}{6} \chi + \frac{1}{8}K^2}$.
\end{proof}

\begin{remark} \label{rank2cor}
The rank 2 analog of Conjecture \ref{conj1} was given in \cite[Conj.~6.7]{GK1}. Assuming the absence of strictly Gieseker semistables sheaves, conjecturally
{\scriptsize{
\begin{align*}
\frac{\sfZ_{S,H,2,c_1}^{\inst}(q,y)}{(y^{\frac{1}{2}} - y^{-\frac{1}{2}})^{\chi}} =&\, 2\Bigg( \frac{1}{2\phi_{-2,1}(q^{\frac{1}{2}},y)^{\frac{1}{2}} \Delta(q^{\frac{1}{2}})^{\frac{1}{2}}} \Bigg)^{\chi} \Bigg( \frac{\theta_3(q,y)+\theta_2(q,y)}{2\eta(q)^2}\Bigg)^{-K^2} \\
&\qquad\qquad \times \sum_{a} \SW(a) \, (-1)^{ac_1} \, \Bigg( \frac{\theta_3(q,y)+\theta_2(q,y)}{\theta_3(q,y)-\theta_2(q,y)} \Bigg)^{aK} \\
&+2 i^{c_1^2} \Bigg( \frac{1}{2\phi_{-2,1}(-q^{\frac{1}{2}},y)^{\frac{1}{2}} \Delta(-q^{\frac{1}{2}})^{\frac{1}{2}}} \Bigg)^{\chi} \Bigg( \frac{\theta_3(q,y)+i\theta_2(q,y)}{2\eta(q)^2}\Bigg)^{-K^2} \\
&\qquad \qquad \times \sum_{a} \SW(a) \, (-1)^{ac_1} \, \Bigg( \frac{\theta_3(q,y)+i\theta_2(q,y)}{\theta_3(q,y)-i\theta_2(q,y)} \Bigg)^{aK},
\end{align*}
}}
where $i = \sqrt{-1}$ and 
\begin{align*}
\theta_3(q,y) = \sum_{n \in \Z} q^{n^2} y^n, \qquad \theta_2(q,y) = \sum_{n \in \Z+ \frac{1}{2}} q^{n^2} y^n.
\end{align*}
We note that $\theta_3(q,y)+\theta_2(q,y)$ is the theta function of the lattice $A_1^\vee$. 
This is a refinement of lines 2+3 of \cite[(5.38)]{VW}, which inspired this formula.
\end{remark}

\begin{remark} \label{strictlyss}
Let $S$ be any smooth projective surface satisfying $b_1=0$ and $p_g>0$. As mentioned earlier, for any $H,r,c_1$, Tanaka-Thomas define $\SU(r)$ Vafa-Witten invariants 
in the presence of Gieseker strictly $H$-semistable Higgs pairs \cite{TT2} (combined with \cite{Tho} for the $y$-refinement). 
We conjecture that the formulae of Corollary \ref{corconj1} and Remark \ref{rank2cor} also holds when there are strictly semistable sheaves. This expectation is based on the fact that Vafa-Witten's original formula \cite[(5.38)]{VW} should hold for \emph{any} $c_1$. However, we have done no verifications when strictly semistable sheaves are present. See \cite{TT2,Tho} for refined/unrefined calculations on $K3$ in the semistable case.
\end{remark}

In Section \ref{chiysec}, we verify Conjecture \ref{conj1}, modulo $x^N$ for some $N$, for:  
\begin{enumerate}
\item $K3$ surfaces, and their blow-ups in one or two points,
\item elliptic surfaces of type $E(3)$, $E(4)$, $E(5)$, and blow-ups of elliptic surfaces of type $E(3)$,
\item double covers of $\PP^2$ branched along a smooth octic and their blow-ups, 
\item double covers of $\PP^1\times\PP^1$ branched along a smooth curve of bidegree $(6,6)$ and their blow-ups,
\item smooth quintic surfaces in $\PP^3$ and their blow-ups.
\end{enumerate}
We also present a numerical version of Conjecture \ref{conj1}, which can be seen as a statement about intersection numbers on Hilbert schemes of points (Section \ref{verif}). This numerical conjecture implies Conjecture \ref{conj1} for surfaces satisfying $b_1=0$, $p_g>0$, and whose only Seiberg-Witten basic classes are $0$ and $K \neq 0$. We test our numerical conjecture in various examples, which include some minimal general type surfaces found by V.~Kanev, F.~Catanese and O.~Debarre, and U.~Persson (Section \ref{verif}).

These verifications are obtained by writing $\chi_{y}^{\vir}(M)$ in terms of (rank 3 descendent) Donaldson invariants of $S$. By Mochizuki's rank 3 formula \cite[Thm.~7.5.2]{Moc}, the latter can be expressed in terms of Seiberg-Witten invariants and integrals over products of Hilbert schemes of points on $S$. 
We show that these integrals are determined by their values on $S = \PP^2$ and $\PP^1 \times \PP^1$, which can be calculated by localization. A similar strategy was employed in the rank 2 case in \cite{GK1, GK2} (which in turn was inspired by \cite{GNY1, GNY3}).

\subsection{Monopole branch} 

For any $H,r,c_1,c_2$, we denote the generating functions of $y$-refined Vafa-Witten invariants by
\begin{align*}
\sfZ_{S,H,r,c_1}(q,y) := q^{-\frac{1}{2r}\chi + \frac{r}{24} K^2} \sum_{c_2} (-1)^{\vd} \, \VW_S^H(r,c_1,c_2,y) \, q^{\frac{\vd}{2r}},
\end{align*}
where $\vd$ is given by \eqref{vd}. We write 
$$
\sfZ_{S,H,r,c_1}(q,y) = \sfZ_{S,H,r,c_1}^{\inst}(q,y) + \sfZ_{S,H,r,c_1}^{\mono}(q,y),
$$
where the first term on the RHS is the contribution from the instanton branch (i.e.~\eqref{instantongenfun}) and the second term is the contribution from the monopole branch. We view these as Fourier expansions in modular parameters 
$$
q=e^{2\pi i \tau}, \quad y=e^{2\pi i z}, \quad (\tau, z) \in \mathfrak{H} \times \C.
$$ 
For all $a, b \in H^2(S,\Z)$, define (suppressing $r$ from the notation)
\begin{equation} \label{deltadef}
\delta_{a,b} := \left\{\begin{array}{cc} 1 & \qquad \textrm{if \ } a-b \in r H^2(S,\Z)  \\ 0 & \qquad \textrm{otherwise}.  \end{array}\right.
\end{equation}
\begin{conjecture} \label{conj2}
Let $S$ be a smooth projective surface satisfying $H_1(S,\Z)= 0$ and $p_g>0$. For any $H,c_1$ we have
{\scriptsize{
\begin{align*}
\frac{\sfZ_{S,H,3,c_1}^{\mono}(q,y)}{(y^{\frac{1}{2}} - y^{-\frac{1}{2}})^{\chi}} = &\Bigg( \frac{1}{\phi_{-2,1}(q^3,y^3)^{\frac{1}{2}} \Delta(q^3)^{\frac{1}{2}}} \Bigg)^{\chi} \Bigg(\frac{\Theta_{A_2,(1,0)}(q^{\frac{1}{2}},y)}{\eta(q)^3}  \Bigg)^{-K^2}  \\
&\qquad\qquad \times\sum_{a,b} \SW(a) \, \SW(b) \,\delta_{c_1+a,b} \, W_+(q^{\frac{1}{2}},y)^{ab} \, W_-(q^{\frac{1}{2}},y)^{(K-a)(K-b)}, 
\end{align*}
}}
\!\!\!\!\! where the sum is over all $(a,b) \in H^2(S,\Z) \times H^2(S,\Z)$ and $W_{\pm}(x,y)$ are the solutions of the following quadratic equations in $\omega$
\begin{align*}
\omega^2 - (W(x,y)^2+3W(x,y)W(x,1)) \, \omega +W(x,y)+3W(x,1) = 0,
\end{align*}
where $W(x,y) := \frac{\Theta_{A_2, (0,0)}(x,y)}{\Theta_{A_2, (1,0)}(x,y)}$.
\end{conjecture}

\begin{remark}
Note that the formula for the instanton branch (Conjecture \ref{conj1}) only features the lattice $A_2^\vee$ whereas the formula for the monopole branch only involves the lattice $A_2$. We will see later that (part of) the instanton branch gets swapped with the monopole branch under $\tau \mapsto -1 / \tau$ (Section \ref{modsec}). Moreover, we have (Lemma \ref{thetalemma})
$$
W(x^3,y) = \frac{Z(x,y)+2}{Z(x,y)-1}.
$$
\end{remark}

\begin{remark} \label{rank2conj}
We have a parallel conjecture in the rank 2 case:
{\scriptsize{
\begin{align*}
\frac{\sfZ_{S,H,2,c_1}^{\mono}(q,y)}{(y^{\frac{1}{2}} - y^{-\frac{1}{2}})^{\chi}} = &\Bigg( \frac{1}{\phi_{-2,1}(q^2,y^2)^{\frac{1}{2}} \Delta(q^2)^{\frac{1}{2}}} \Bigg)^{\chi} \Bigg( \frac{\theta_3(q,y)}{\eta(q)^2} \Bigg)^{-K^2} (-1)^\chi \sum_{a} \SW(a) \, \delta_{c_1,a} \,  \Bigg( \frac{\theta_3(q,y)}{\theta_2(q,y)} \Bigg)^{a K},
\end{align*}
}}
where we note that $\theta_3(q,y)$ is the theta function of the $A_1$ lattice.
\end{remark}

Altogether \ref{rank2cor}+\ref{rank2conj} and  \ref{corconj1}+\ref{conj2} provide closed conjectural formulae for the $y$-refined $\SU(2)$ and $\SU(3)$ Vafa-Witten invariants of any polarized surface $(S,H)$ satisfying $b_1=0$ and $p_g>0$, and any $c_1$. 
We explore some consequences of these formulae, e.g.~to blow-ups, in Section \ref{conseq}.

\begin{remark}
For any prime rank $r>2$, there exists a formula for $\mathrm{SU}(r)$ Vafa-Witten invariants in the physics literature \cite[(5.13)]{LL}. This formula supposedly applies to any smooth projective surface $S$ such that $H_1(S,\Z) = 0$ and $|K|$ contains a smooth connected curve. However, this formula is incorrect as can be seen from the following example. Let $S \rightarrow \PP^1$ be an elliptic surface with section, 36 rational nodal fibres, and no further singular fibres.
Let $F$ be the class of a fibre and $B$ the class of a section. Then $|K| = |F|$ and taking $c_1 = B$, Labastida-Lozano's formula reduces to zero. However, taking $c_2=3$ and a suitable polarization $H$, a result of T.~Bridgeland \cite{Bri} implies that $M_S^H(3,B,3)$ is smooth of expected dimension and consists of a single reduced point, so $e^{\vir}(M) = e(M) = 1$ (consistent with Conjecture \ref{conj1}). 
\end{remark}

We have the following evidence for Conj.~\ref{conj2} and Remark \ref{rank2conj} (Section \ref{monsec}):
\begin{itemize}
\item Let $(S,H)$ be a polarized surface satisfying $b_1=0$, $|K|$ contains a smooth connected curve, and any line bundle $L$ on $S$ satisfying $0 \leq \deg L \leq \frac{1}{2} \deg K$ is trivial. Take $c_1 = K$. Then the monopole branch of $N_S^H(2,c_1,c_2)^{\C^*}$ is smooth if and only if $c_2 = 0,1,2,3$ \cite{TT1}. For $c_2=0,1,2$, Thomas calculates the monopole contribution to $\VW_S^H(2,c_1,c_2,y)$ \cite{Tho}. His result matches the prediction of Remark \ref{rank2conj}.
\item Suppose $S,H,c_1$ are as in the previous item. The monopole branch of $N_S^H(3,c_1,c_2)^{\C^*}$ is smooth if and only if $c_2 = 0,1,2$ \cite{Laa}. For $c_2=0,1,2$, T.~Laarakker determines the monopole contribution to $\VW_S^H(3,c_1,c_2,y)$. His result matches the prediction of Conjecture \ref{conj2}.
\item Let $(S,H)$ be polarized surface satisfying $H_1(S,\Z)=0$ and  $p_g>0$. Let $r,c_1$ be chosen such that $r$ is prime and there exist no rank $r$ strictly Gieseker $H$-semistable Higgs pairs on $S$ with first Chern class $c_1$. Then A.~Gholampour and Thomas \cite{GT1, GT2} express the monopole contribution to the Vafa-Witten invariants in terms of (virtual) intersection numbers on nested Hilbert schemes of points and curves on $S$ (see also \cite{GSY,Tho}). Based on this result, Laarakker expresses $\sfZ_{S,H,r,c_1}^{\mono}(q,y)$ in terms of Seiberg-Witten invariants and universal power series, which can be written in terms of intersection numbers on $S^{[n_1]} \times \cdots \times S^{[n_r]}$. The latter are entirely determined on $S = \PP^2, \ \PP^1 \times \PP^1$ much like in Section \ref{chiysec} of this paper. Localization calculations allow him to verify Remark \ref{rank2conj} and Conjecture \ref{conj2} up to certain orders (Section \ref{monsec}). 
\end{itemize}

\subsection{Refined modularity} \label{modintro}

Let $r=1$ or $r>1$ prime. Assume $H_1(S,\Z) = 0$. Motivated by $S$-duality, physicists predict that $\sfZ_{S,H,r,c_1}(q)$ only depends on $[c_1] \in H^2(S,\Z) / r H^2(S,\Z)$ and is the Fourier expansion of a meromorphic function $\sfZ_{S,H,r,c_1}(\tau)$ on $\mathfrak{H}$ satisfying
 \cite[(5.39)]{VW}, \cite[(5.22)]{LL}
\begin{align} \label{modunref}
\begin{split}
\sfZ_{S,H,r,c_1}(\tau+1) &= (-1)^{r \chi} \, e^{\frac{\pi i r}{12} K^2} \, e^{-\frac{\pi i (r-1)}{r} c_1^2} \, \sfZ_{S,H,r,c_1}(\tau), \\
\sfZ_{S,H,r,c_1}(-1/\tau) &= (-1)^{(r-1)\chi} \, r^{1-\frac{e}{2}} \, \Big( \frac{\tau}{i} \Big)^{-\frac{e}{2}} \, \sum_{[a]} e^{\frac{2 \pi i}{r} c_1 a} \sfZ_{S,H,r,a}(\tau),
\end{split}
\end{align}
where the sum is over all $[a] \in H^2(S,\Z) / r H^2(S,\Z)$. We refer to the second transformation in \eqref{modunref} as the $S$-duality transformation. 

\begin{remark}
There is a subtlety in the interpretation of these statements (and the statement of the conjecture below). More precisely: conjecturally there exists a series $\widetilde{\sfZ}_{S,H,r,c_1}(q)$ defined for any $S,H,r$ as above and any \emph{possibly non-algebraic} $c_1 \in H^2(S,\Z)$ such that:
\begin{itemize}
\item $\widetilde{\sfZ}_{S,H,r,c_1}(q)$ only depends on $[c_1] \in H^2(S,\Z) / rH^2(S,\Z)$,
\item $\widetilde{\sfZ}_{S,H,r,c_1}(q) = \sfZ_{S,H,r,c_1}(q)$ for algebraic classes $c_1 \in H^2(S,\Z)$,
\item $\widetilde{\sfZ}_{S,H,r,c_1}(q)$ is the Fourier expansion of a meromorphic function $\widetilde{\sfZ}_{S,H,r,c_1}(\tau)$ on $\mathfrak{H}$ satisfying \eqref{modunref}.
\end{itemize}
Indeed, after multiplying by $(y^{\frac{1}{2}} - y^{-\frac{1}{2}})^\chi$ and setting $y=1$, the expression for $\widetilde{\sfZ}_{S,H,2,c_1}(q)$ is obtained by summing the RHS of \ref{rank2cor}+\ref{rank2conj}, which makes sense for \emph{any} $c_1 \in H^2(S,\Z)$ and which only depends on $[c_1] \in H^2(S,\Z) / 2 H^2(S,\Z)$. Similarly for $\widetilde{\sfZ}_{S,H,3,c_1}(q)$ using \ref{corconj1}+\ref{conj2}. 
\end{remark}

We conjecture the following $y$-refinement of \eqref{modunref}:
\begin{conjecture} \label{conj3}
Let $S$ be a smooth projective surface satisfying $H_1(S,\Z) = 0$ and $p_g>0$. Let $H$ be a polarization on $S$, $r=1$ or $r$ prime, and $c_1 \in H^2(S,\Z)$. Then $\sfZ_{S,H,r,c_1}(q,y)$ only depends on $[c_1] \in H^2(S,\Z) / r H^2(S,\Z)$ and is the Fourier expansion of a meromorphic function $\sfZ_{S,H,r,c_1}(\tau,z)$ on $\mathfrak{H} \times \C$ satisfying
\begin{align} \label{modref}
\begin{split}
\sfZ_{S,H,r,c_1}(\tau,z)\Big|_{(\tau+1,z)} =& \, (-1)^{r \chi} e^{\frac{\pi i r}{12} K^2} e^{-\frac{\pi i (r-1)}{r} c_1^2} \sfZ_{S,H,r,c_1}(\tau,z), \\
\frac{\sfZ_{S,H,r,c_1}(\tau,z)}{(y^{\frac{1}{2}} - y^{-\frac{1}{2}})^{\chi}}\Big|_{(-1/\tau,z/\tau)} =& \, (-1)^{r \chi} r^{1-\frac{e}{2}} i^{-\frac{K^2}{2}} \tau^{-5\chi+\frac{K^2}{2}} e^{\frac{2 \pi i z^2}{\tau}\Big(-\frac{r}{2} \chi - \frac{r(r^2-1)}{24} K^2 \Big)} \\
& \, \times \sum_{[a]} e^{\frac{2 \pi i}{r} c_1 a} \frac{\sfZ_{S,H,r,a}(\tau,z)}{(y^{\frac{1}{2}} - y^{-\frac{1}{2}})^{\chi}},
\end{split}
\end{align}
where the sum is over all $[a] \in H^2(S,\Z) / r H^2(S,\Z)$.
\end{conjecture}

In Section \ref{modsec}, we provide the following evidence for this conjecture:
\begin{itemize}
\item Conjecture \ref{conj3} holds for $r=1$ and implies \eqref{modunref}.
\item For $r$ prime, we conjecture a formula for $\sfZ_{K3,H,r,c_1}(\tau,z)$, refining an existing formula for $\sfZ_{K3,H,r,c_1}(\tau)$, which satisfies Conjecture \ref{conj3}.
\item Assume Remarks \ref{rank2cor}, \ref{strictlyss}, \ref{rank2conj}. Then Conjecture \ref{conj3} holds for $r=2$.
\item Assume Conj.~\ref{conj1}, \ref{conj2}, and Rem.~\ref{strictlyss}. Then Conj.~\ref{conj3} holds for $r=3$.
\end{itemize}

\noindent \textbf{Acknowledgements.} We warmly thank Richard Thomas for providing drafts of \cite{Tho}. We also thank Ties Laarakker for crucial discussions. In the early stages of the project, there was no direct evidence for our formulae on the monopole branch (Conjecture \ref{conj2}, Remark \ref{rank2conj}) other than the fact that the unrefined formulae transform according to the physicists' predictions \eqref{modunref}. However, around the same time Maulik-Thomas \cite{MT,Tho} defined the monopole contribution to the $y$-refined VW invariants and Thomas \cite{Tho} calculated the monopole contribution for rank 2 and $c_2=0,1,2$, which were consistent with our predictions, and later Laarakker \cite{Laa} extended these calculations to rank 3 and more values of $c_2$ (Section \ref{monsec}). This in turn enabled us to present  ``$y$-refined modularity'' as a compelling separate conjecture (Conjecture \ref{conj3}). 

We also thank Jan Manschot for useful discussions and for pointing out an error in the coefficient of $(2 \pi i z^2 / \tau)K^2$ in formula \eqref{modref} in a previous version of this paper. See \cite{AMP}.

This material is based upon work supported by the National Science Foundation under Grant No.~DMS-1440140 while M.K. was in residence at the Mathematical Sciences Research Institute in Berkeley, California, during the Spring 2018 Semester.

\section{Instanton branch} \label{chiysec}

\subsection{Descendent Donaldson invariants} \label{Donaldson} Let $S$ be a smooth projective surface satisfying $b_1 = 0$. For a polarization $H$, we denote by $M:=M_S^H(r,c_1,c_2)$ the moduli space of rank $r$ Gieseker $H$-stable sheaves on $S$ with Chern classes $c_1,c_2$. We assume there are no rank $r$ strictly Gieseker $H$-semistable sheaves on $S$ with Chern classes $c_1,c_2$. 
Consider the projections
\begin{displaymath}
\xymatrix
{
& M \times S \ar_{\pi_M}[dl] \ar^{\pi_S}[dr] & \\
M & & S  
}
\end{displaymath}
The moduli space $M$ admits a perfect obstruction theory with virtual tangent bundle \cite{Moc}
$$
T^\vir = R\hom_\pi(\E,\E)_0[1],
$$
where $\E$ denotes the universal sheaf on $M \times S$, $(\cdot)_0$ denotes trace-free part, and $R\hom_{\pi_M}(\cdot, \cdot) := R\pi_{M*} R\hom(\cdot,\cdot)$. Although the universal sheaf $\E$ may only exist \'etale locally, the complex $T^\vir$ always exists globally \cite[Sect.~10.2]{HL}. We have a corresponding virtual cycle
$$
[M]^\vir \in H_{2\vd}(M),
$$
where $\vd = \vd(M)$ is given by \eqref{vd}.

Next, we assume the universal sheaf $\E$ exists \emph{globally} on $M \times S$.\footnote{We will get rid of this assumption in Remark \ref{univexists}.} For any $\sigma \in H^*(S,\Q)$ and $\alpha \geq 0$, we define the descendent insertion
$$
\tau_{\alpha}(\sigma) := \pi_{M*} \big( \ch_{2+\alpha}(\E) \cap \pi_{S}^{*} \, \sigma \big).
$$
Let $P(\E)$ be any polynomial in descendent insertions. Then we refer to
$$
\int_{[M]^{\vir}} P(\E) \in \Q
$$
as a (descendent, algebraic) Donaldson invariant. These Donaldson invariants were studied in depth by T.~Mochizuki. In \cite{GK1}, we observe that $\chi_y^\vir(M)$ can be expressed in terms of Donaldson invariants. We recall the precise statement.

Let $X$ be a projective $\C$-scheme and denote by $K^0(X)$ the $K$-group generated by locally free sheaves on $X$. For any vector bundle $E$ on $X$, define
$$
\Lambda_y E := \sum_{p=0}^{\rk(E)} [\Lambda^p E] \, y^p \in K^0(X)[[y]].
$$
See \cite{FG} for an extension of this definition to arbitrary elements of $K^0(X)$. Furthermore, for any element $E$ of $K^0(X)$ we set
\begin{align} \label{defT}
\sfT_{y}(E,t) := t^{-\rk E} \sum_k \left\{ \ch(\Lambda_{y} E^{\vee}) \, \td(E) \right\}_k t^{k},
\end{align}
where $\{\cdot\}_k \in A^k(X)_{\Q}$ denotes the degree $k$ part in the Chow ring. We have the following list of basic properties \cite{GK1}:
\begin{itemize}
\item $\sfT_y(E_1 + E_2,t) = \sfT_y(E_1,t) \, \sfT_y(E_2,t)$ for all $E_1, E_2 \in K^0(X)$,
\item $\sfT_y(L,t) = \frac{x (1+y e^{-x t})}{1-e^{-x t}}$ for any line bundle $L$ on $X$ with $c_1(L) = x$,
\item $\sfT_{y}(E,1+y) \in \Q[1+y]$ for all $E \in K^0(X)$,
\item $\sfT_{y}(E - \O_X^{\oplus r},1+y) = \sfT_{y}(E,1+y)$ for all $E \in K^0(X)$ and $r\geq0$, 
\item $\sfT_{y}(E,1+y) \Big|_{y=-1} = c(E)$, i.e.~the total Chern class of $E \in K^0(X)$. 
\end{itemize}
Using virtual Hirzebruch-Riemann-Roch, Grothendieck-Riemann-Roch, and K\"unneth decomposition, one can show the following \cite[Prop.~2.1]{GK1}:
\begin{proposition} \label{chiyinsert}
For $S,H,r,c_1,c_2$ as above, there exists a polynomial expression $P(\E)$ in certain descendent insertions $\tau_{\alpha}(\sigma)$ and $y$ such that
$$
\chi^{\vir}_{-y}(M_{S}^{H}(r,c_1,c_2)) = \int_{[M_{S}^{H}(r,c_1,c_2)]^{\vir}} P(\E).
$$
\end{proposition}

\subsection{Mochizuki's rank 3 formula}

In his remarkable book \cite{Moc}, Mochizuki derives a formula for (descendent, algebraic) Donaldson invariants \emph{for any rank} in terms of Seiberg-Witten invariants and integrals over Hilbert schemes of points \cite[Thm.~7.5.2]{Moc}. For rank $r=2$, his formula has interesting applications to Witten's conjecture \cite{GNY3}, $\SU(2)$ Vafa-Witten invariants \cite{GK1}, and refinements thereof \cite{GK1, GK2}. In this paper, we apply Mochizuki's formula for rank $r=3$ to $y$-refined $\SU(3)$ Vafa-Witten invariants.

Let $S$ be a smooth projective surface satisfying $b_1 = 0$ and $p_g > 0$. As in the introduction, we denote the Seiberg-Witten invariants of $S$ by $\SW(a)$ and the Hilbert scheme of points $n$ points on $S$ by $S^{[n]}$. The latter has a universal subscheme
\begin{displaymath}
\xymatrix
{
& \cZ \ar@{^(->}[r] \ar_q[dl] \ar^p[dr] & S \times S^{[n]}    \\
S & & S^{[n]}.
}
\end{displaymath}
We denote the universal ideal sheaf by $\I_{\cZ}$. When $L$ is a line bundle on $S$, we denote the corresponding rank $n$ tautological vector bundle by
$
L^{[n]} := p_* q^* L.
$

Consider a product of three Hilbert schemes
\begin{equation} \label{tripleHilb}
S^{[n_1]} \times S^{[n_2]} \times S^{[n_3]}.
\end{equation}
Denote the pull-backs of the various universal ideal sheaves on $S \times \prod_{i=1}^{3} S^{[n_i]}$ by $\I_1$, $\I_2$, $\I_3$. We endow \eqref{tripleHilb} with a trivial $\C^* \times \C^*$ action and we denote the generators of the corresponding character group by $\s_1, \s_2$. Moreover 
$$
H^*(B(\C^* \times \C^*),\Q) = H^*_{\C^* \times \C^*}(pt,\Q) \cong \Q[s_1,s_2],
$$
where 
$$
s_1 = c_1^{\C^* \times \C^*}(\s_1), \quad s_2 = c_1^{\C^* \times \C^*}(\s_2)
$$
are the corresponding equivariant parameters. For later use, we introduce ``characters''\footnote{These are elements of $X(\C^* \times \C^*) \otimes_\Z \Q$, where $X(\C^* \times \C^*)$ denotes the character lattice.}
\begin{equation} \label{Tfrak}
\mathfrak{T}_1  = \s_1^{-1} \quad \mathfrak{T}_2 = \s_1^{\frac{1}{2}} \s_2^{-1} \quad \mathfrak{T}_3 = \s_1^{\frac{1}{2}} \s_2
\end{equation}
and we define $T_i := c_1^{\C^* \times \C^*}(\mathfrak{T}_i)$.

Some more notation. Let $a \in A^1(S)$ be a divisor class on $S$, then we denote the corresponding line bundle (up to isomorphism) by $\O(a)$. Furthermore
$$
\chi(a) := \frac{a^2 - aK}{2} + \chi,
$$
where $K:=K_S$ and $\chi:=\chi(\O_S)$. Furthermore, for $\ch = (r,c_1,\frac{1}{2}c_1^2 - c_2) \in \bigoplus_i H^{2i}(S,\Q)$, we write
$$
\chi(\ch) :=  \int_S \ch \cdot \td(S).
$$
When $H$ is a polarization on $S$, we denote the reduced Hilbert polynomials associated to $a$ and $\ch$ by (provided $r>0$)
$$
p_{a}(m) := \chi(e^{mH+a}) = \chi(mH+a) \qquad p_{\ch}(m) := \chi(e^{mH} \cdot \ch) / r.
$$

Let $P(\E)$ be \emph{any} polynomial in descendent insertions $\tau_{\alpha}(\sigma)$ arising from a polynomial in Chern numbers of $T^\vir$ (e.g.~like in Prop.~\ref{chiyinsert}). For any $a_1, a_2 \in A^1(S)$ and $n_1, n_2 \in \Z_{\geq 0}$, define $\Psi(a_1,a_2,a_3,n_1,n_2,n_3)$ by the following expression
{\scriptsize{
\begin{align} 
\begin{split} \label{Psi}
&\Res_{s_2 = 0} \Res_{s_1 = 0} \Bigg( P\big(\bigoplus_{i=1}^{3} \I_i(a_i) \otimes \mathfrak{T}_i\big)  \, \prod_{i=1}^{2} \Eu(\O(a_i)^{[n_i]}) s_i^{-1+\sum_{j \geq i} \chi(y_j)} \prod_{1\leq i<j \leq 3} \frac{\Eu(\O(a_j)^{[n_j]} \otimes \mathfrak{T}_j \otimes \mathfrak{T}_i^{-1})}{(T_j-T_i)^{\chi(a_j)}   Q(\I_i(a_i) \otimes \mathfrak{T}_i, \I_j(a_j) \otimes \mathfrak{T}_j) } \Bigg) \\
&\mathrm{where} \ y_i = (1,a_i,\frac{1}{2}a_i^2 - n_i).
\end{split}
\end{align}
}}

\noindent We explain the notation. In this formula, $\I_i(a_i)$ is short-hand for $\I_i \otimes \pi_S^* \O(a_i)$. Furthermore, $\Eu(\cdot)$ denotes $\C^* \times \C^*$ equivariant Euler class. Note that $\I_i(a_i)$ and $\O(a_i)^{[n_i]}$ have trivial $\C^* \times \C^*$ equivariant structures, so the equivariant structures come entirely from the characters $\mathfrak{T}_i$. Next, $\Res_{s_i=0}(\cdot)$ is the residue at $s_i = 0$, i.e.~the coefficient of $s_i^{-1}$ of $(\cdot)$ viewed as a Laurent series in $s_i$. For any $\C^* \times \C^*$ equivariant sheaves $\E_1$, $\E_2$ on $S \times \prod_{i=1}^{3} S^{[n_i]}$, flat over $\prod_{i=1}^{3} S^{[n_i]}$, define
$$
Q(\E_1,\E_2) :=\Eu(- R\hom_{\pi}(\E_1,\E_2) - R\hom_{\pi}(\E_2,\E_1)), 
$$
where $\pi : S \times \prod_{i=1}^{3} S^{[n_i]} \rightarrow \prod_{i=1}^{3} S^{[n_i]}$ denotes projection.
Moreover, $P(\cdot)$ is the expression obtained from $P(\E)$ by formally replacing $\E$ by $\cdot$. For later use, we define
$$
\widetilde{\Psi}(a_1,a_2,a_3,n_1,n_2,n_3,s_1,s_2)
$$
by expression \eqref{Psi} but \emph{without} applying $\Res_{s_2 = 0} \Res_{s_1 = 0}$. 

Fix a Chern character $\ch = (3,c_1,\frac{1}{2}c_1^2 - c_2)$. For any decomposition $c_1 = a_1 + a_2 + a_3 \in A^1(S)$, define
\begin{equation*} 
\cA(a_1,a_2,a_3,c_2) := \sum_{n_1 + n_2 + n_3 = c_2 - \sum_{i<j} a_i a_j} \int_{\prod_{i=1}^{3} S^{[n_i]}} \Psi(a_1,a_2,a_3,n_1,n_2,n_3).
\end{equation*}
Denote the same expression, with $\Psi$ replaced by $\widetilde{\Psi}$, by $\widetilde{\cA}(a_1,a_2,a_3,c_2,s_1,s_2)$.

\begin{theorem}[Mochizuki] \label{mocthm}
Let $S$ be a smooth projective surface satisfying $b_1 = 0$ and $p_g >0$. Let $H, c_1,c_2$ be chosen such that there exist no rank 3 strictly Gieseker $H$-semistable sheaves on $S$ with Chern classes $c_1,c_2$. Suppose the following conditions hold:
\begin{enumerate}
\item[(i)] There exists a universal sheaf $\E$ on $M_{S}^{H}(3,c_1,c_2) \times S$.
\item[(ii)] $\chi(\ch) > \chi$, where $ \chi:=\chi(\O_S)$.
\item[(iii)] $p_{\ch} > p_{K}$.
\item[(iv)] For all Seiberg-Witten basic classes $a_1, a_2, a_3$ satisfying $a_2H \leq a_3H$ and $a_1 H \leq \frac{1}{2}(a_2+a_3)H$, both inequalities are strict. 
\end{enumerate}
Let $P(\E)$ be any polynomial in descendent insertions, which arises from a polynomial in Chern numbers of $T^\vir$ (e.g.~like in Prop.~\ref{chiyinsert}). Then\footnote{Our formula differs by a factor $3$ from Mochizuki's. Mochizuki works on the moduli stack of oriented sheaves which maps to $M$ via a degree $\frac{1}{3}:1$ \'etale morphism.}
$$
\int_{[M_{S}^{H}(3,c_1,c_2)]^{\vir}} P(\E) = 3  \sum_{{\scriptsize{\begin{array}{c} c_1 = a_1 + a_2 + a_3 \\ a_1 H < \frac{1}{2}(a_2+a_3)H \\ a_2 H < a_3 H \end{array}}}} \SW(a_1) \, \SW(a_2) \, \cA(a_1,a_2,a_3,c_2).
$$
\end{theorem}

\begin{remark} \label{univexists}
In this theorem, assumption (i) can be dropped. Since $\E$ always exists \'etale locally, the complex $T^{\vir} = -R\hom_{\pi}(\E,\E)_0$ exists globally so the left-hand side of Mochizuki's formula makes sense. Furthermore, Mochizuki \cite{Moc} works over the Deligne-Mumford stack of oriented sheaves,
which always has a universal sheaf. From this it can be seen that global existence of the universal sheaf can be omitted from the assumptions. Another advantage, when working on the stack, is that $P$ can be \emph{any} polynomial in descendent insertions defined using the universal sheaf of the stack.
\end{remark}

\begin{remark} \label{assumpmocthm}
In \cite{GNY3}, the authors conjecture that assumptions (iii) and (iv) can be dropped and the sum can be replaced by a sum over \emph{all} Seiberg-Witten basic classes. Assumption (ii) is necessary. We call this the strong form of Mochizuki's formula.
\end{remark}

\subsection{Eleven universal functions}

In this section, we want to isolate the part of Mochizuki's formula (Theorem \ref{mocthm}), which involves integrals over Hilbert schemes of points. These are best studied by combining them into a generating function. 

Let $S$ be \emph{any} smooth projective surface. We recall that the tangent bundle to the Hilbert scheme satisfies
$$
T_{S^{[n]}} \cong R\hom_\pi(\I,\I)_0[1],
$$
where $\pi : S \times S^{[n]} \rightarrow S^{[n]}$ denotes projection. Furthermore, on the triple product of Hilbert schemes we use the projections
\begin{displaymath}
\xymatrix
{
& S^{[n_1]} \times S^{[n_2]} \times S^{[n_3]} & \\
& S \times S^{[n_1]} \times S^{[n_2]} \times S^{[n_3]} \ar_{\pi}[u] \ar_{\pi_1}[dl] \ar^{\pi_2}[d] \ar^{\pi_3}[dr] & \\
S^{[n_1]} & S^{[n_2]} & S^{[n_3]}.
}
\end{displaymath}

\begin{definition}
Let $a_1,a_2,a_3 \in A^1(S)$. Define:
\begin{align*}
&\sfZ_S(a_1,a_2,a_3,s_1,s_2,y,q):= \sum_{n_1,n_2,n_3 \geq 0} (q/s_1)^{n_1} (q/(s_1s_2))^{n_2} (q/(s_1s_2))^{n_3} \times \\
&\int_{\prod_{i=1}^{3} S^{[n_i]}} \frac{\sfT_{-y}^{\C^* \times \C^*}(E_{n_1,n_2,n_3},1-y)}{\Eu(E_{n_1,n_2,n_3} - \sum_{i=1}^{3} \pi_i^* T_{S^{[n_i]}})} \, \prod_{i=1}^{2} \Eu(\O(a_i)^{[n_i]}) \prod_{i<j} \Eu(\O(a_j)^{[n_j]} \otimes \mathfrak{T_j} \otimes \mathfrak{T}_i^{-1}).
\end{align*} 
Here $\sfT_{-y}^{\C^* \times \C^*}$ denotes the $\C^* \times \C^*$ equivariant analog of \eqref{defT}, $\pi_i$ denote the projections from the various factors of $\prod_{i=1}^{3} S^{[n_i]}$, and
\begin{align*}
E_{n_1,n_2,n_3}:= \sum_{i=1}^{3} \pi_i^* T_{S^{[n_i]}}  + \sum_{i \neq j} \Bigg( \chi(a_j-a_i) \otimes \O - R\hom_\pi(\I_i(a_i),\I_j(a_j)) \Bigg) \otimes \mathfrak{T}_j \otimes \mathfrak{T}_i^{-1},
\end{align*}
was $\mathfrak{T}_i$ were defined in \eqref{Tfrak} in terms of $\s_1,\s_2$. The generating function $\sfZ_S$ is normalized, i.e. it satisfies
$$
\sfZ_S(a_1,a_2,a_3,s_1,s_2,y,q) \in 1 + q \, \Q(\!(s_1,s_2)\!)[y][[q]].
$$
We define the following normalization term
\begin{align*}
n_S(a_1,a_2,a_3,s_1,s_2) = &\, s_1^{-1 + \sum_{i \geq 1} \chi(a_i)} s_2^{-1 + \sum_{i \geq 2} \chi(a_i)} \prod_{1\leq i<j \leq 3} \frac{1}{(T_j-T_i)^{\chi(a_j)}} \\
&\times \prod_{i \neq j} \frac{1}{\sfT_{-y}^{\C^* \times \C^*}(\chi(a_j-a_i) \otimes \mathfrak{T}_j \otimes \mathfrak{T}_i^{-1},1-y)}. 
\end{align*}
\end{definition}

Let $S$ be a surface satisfying $b_1 = 0$, $p_g > 0$, and suppose the assumptions of Theorem \ref{mocthm} are satisfied. Then $\chi_{-y}(M_S^H(3,c_1,c_2))$ is given by the coefficient of $x^{\vd}$ of the following series
\begin{align}
\begin{split} \label{keyeqn}
&3\sum_{{\scriptsize{\begin{array}{c} c_1 = a_1 + a_2 + a_3 \\ a_1 H < \frac{1}{2}(a_2+a_3)H \\ a_2 H < a_3 H \end{array}}}} \SW(a_1) \, \SW(a_2) \, \Res_{s_2=0} \, \Res_{s_1=0}  \\
&\qquad\times x^{-8 \chi + 2\sum_{i<j} a_i a_j - 2 \sum_i a_i^2} \, n_S(a_1,a_2,a_3,s_1,s_2) \, \sfZ_S(a_1,a_2,a_3,s_1,s_2,x^6).
\end{split}
\end{align}

Let us go back to an arbitrary surface $S$ and arbitrary $a_1,a_2,a_3 \in A^1(S)$. Then $\sfZ_S(a_1,a_2,a_3,s_1,s_2,q)$ has two significant properties:
\begin{itemize}
\item As a power series in $q$, the coefficients of $\sfZ_S(a_1,a_2,a_3,s_1,s_2,q)$ are universal polynomials in $$a_1^2,a_1a_2,a_2^2,a_2a_3,a_3^2,a_1a_3,a_1K,a_2K,a_3K,K^2,\chi,$$ where $K:=K_S$ and $\chi:=\chi(\O_S)$. This essentially follows from \cite{EGL}.
\item Suppose $S = S' \sqcup S^{''}$ is a disjoint union of two surfaces and $a_i = a_i' + a_i''$ with $a_i \in A^1(S')$, $a_i''\in A^1(S'')$. Then
$$
\sfZ_S(a_1,a_2,a_3,s_1,s_2,q) = \sfZ_{S'}(a_1',a_2',a_3',s_1,s_2,q) \sfZ_{S''}(a_1'',a_2'',a_3'',s_1,s_2,q).
$$
This essentially follows from the property $\sfT_y(E_1 + E_2,t) = \sfT_y(E_1,t) \sfT_y(E_2,t)$ discussed in Section \ref{Donaldson}.
\end{itemize}

These facts imply the following proposition (for the proof, see  \cite[Prop.~3.3]{GK1}).
\begin{proposition} \label{univ}
There exist universal functions 
$$
A_1(s_1,s_2,y,q), \ldots, A_{11}(s_1,s_2,y,q) \in 1 + q \, \Q[y](\!(s_1,s_2)\!)[[q]]
$$ 
such that for any smooth projective surface $S$ and $a_1, a_2, a_3 \in A^1(S)$ we have
$$
\sfZ_{S}(a_1,a_2,a_3,s_1,s_2,y,q) =  A_1^{a_1^2} A_2^{a_1a_2} A_3^{a_2^2} A_4^{a_2a_3} A_5^{a_3^2} A_6^{a_1a_3} A_7^{a_1K} A_8^{a_2K} A_9^{a_3K} A_{10}^{K^2} A_{11}^{\chi}.
$$
\end{proposition}

\subsection{Computer verifications} \label{verif}

Let $S$ be a smooth projective surface satisfying $b_1 = 0$ and $p_g>0$. When the assumptions of Theorem \ref{mocthm} are satisfied, formula  \eqref{keyeqn} expresses $\chi_{-y}^{\vir}(M_S^H(3,c_1,c_2))$ in terms of $\SW(a)$ and $\sfZ_S(a_1,a_2,a_3,s_1,s_2,q)$. Seiberg-Witten invariants of algebraic surfaces satisfying $b_1 = 0$ and $p_g>0$ are often rather easy to calculate. E.g.~when $S$ is minimal of general type, the Seiberg-Witten basic classes are $0$ and $K$ and 
$$
\SW(0) = 1, \qquad \SW(K) = (-1)^{\chi}.
$$ 
The generating function $\sfZ_S(a_1,a_2,a_3,s_1,s_2,q)$ is determined by eleven universal functions $A_i$ (Proposition \ref{univ}). 

Since $\sfZ_S(a_1,a_2,a_3,s_1,s_2,q)$ is defined for \emph{any} surface $S$, the universal functions are determined by
\begin{align*}
(S,a_1,a_2,a_3) = &(\PP^2,\O,\O,\O), \\ 
& (\PP^1 \times \PP^1,\O,\O,\O), \\
&(\PP^1 \times \PP^1,\O(-1,0),\O,\O), \\ 
&(\PP^1 \times \PP^1,\O,\O(-1,0),\O), \\ 
&(\PP^1 \times \PP^1,\O,\O,\O(-1,0))\\
& (\PP^1 \times \PP^1,\O(-1,1),\O,\O), \\
& (\PP^1 \times \PP^1,\O,\O(-1,1),\O), \\
& (\PP^1 \times \PP^1,\O,\O,\O(-1,1)), \\
&(\PP^1 \times \PP^1,\O(-1,0),\O(0,-1),\O), \\ 
&(\PP^1 \times \PP^1,\O(-1,0),\O,\O(0,-1)), \\
&(\PP^1 \times \PP^1,\O,\O(-1,0),\O(0,-1)).
\end{align*}

Then $(a_1^2,a_1a_2,a_2^2,a_2a_3,a_3^2,a_1a_3,a_1K,a_2K,a_3K,K^2,\chi)$ determines an $11 \times 11$ invertible matrix. Moreover, we can use the $\C^* \times \C^*$ torus action and Atiyah-Bott localization to express $\sfZ_S(a_1,a_2,a_3,s_1,s_2,q)$ in terms of sums over torus fixed points of $\prod_{i=1}^{3} S^{[n_i]}$, which are indexed by partitions. This turns the calculation of $\sfZ_S(a_1,a_2,a_3,s_1,s_2,q)$ into a purely combinatorial problem, which can be implemented in Maple or Pari/GP. For details see \cite[Sect.~4]{GK1}. This allows us to determine $A_i(s_1,s_2,y,q)$ up to the following orders in $s_1,s_2,y,q$:

\begin{itemize}
\item $A_i(s_1,s_2,1,q)$ up to (and including) orders $(35,40,12)$ in $(s_1,s_2,q)$, 
\item  $A_i(s_1,s_2,y,q)$ up to (and including) orders $(19,24,2,6)$ in $(s_1,s_2,y,q)$.
\end{itemize}

\begin{remark}
If $M$ is a proper $\C$-scheme with perfect obstruction theory of virtual dimension $d$, then $\chi_{-y}^{\vir} \in \Z[y]$ has degree $\leq d$ and satisfies $\chi_{-y}^{\vir}(M) = y^d \chi_{-1/y}^{\vir}(M)$ \cite[Thm.~4.5, Rem.~4.13]{FG}. Therefore
\begin{align*}
\chi_{-y}^{\vir}(M) \mod y^3, \qquad e^{\vir}(M) = \chi_{-1}(M)
\end{align*}
determine $\chi_{-y}^{\vir}(M)$ when $d \leq 7$. 
\end{remark}

We use this data to verify Conjecture \ref{conj1} in various examples by using Theorem \ref{mocthm}, equation \eqref{keyeqn}, and Proposition \ref{univ}. Suppose $S,H,c_1,c_2$ are chosen such that there exist no rank 3 strictly Gieseker $H$-semistable sheaves on $S$ with Chern classes $c_1,c_2$ and suppose 
\begin{equation*}
c_2 < \frac{1}{2}c_1(c_1-K) + 2\chi,
\end{equation*}
which is condition (ii) of Theorem \ref{mocthm}. In most examples, we also assume
\begin{equation} \label{polarineq}
\frac{1}{3} Hc_1> HK,
\end{equation}
which, in these examples, implies conditions (iii) and (iv) of Theorem \ref{mocthm}. In some examples, indicated by $\star$, we do not assume \eqref{polarineq} in which case we assume Remark \ref{assumpmocthm} holds (strong form of Mochizuki's formula). 

We verified Conjecture \ref{conj1} in the following cases:
\begin{enumerate}
\item $S=K3$ and
\begin{itemize}
\item $\vd \leq 8$, 
\item $y=1$ and $\vd \leq 20$,
\end{itemize}
\item $S$ is $K3$ blown up in a point and 
\begin{itemize}
\item $\vd \leq 8$, 
\item $y=1$, and $\vd \leq 20$,
\end{itemize}
\item $S$ is $K3$ blown up in two points and 
\begin{itemize} 
\item $\vd \leq 10$,
\item $y=1$, and $\vd \leq 14$,
\end{itemize}
\item $S$ is an elliptic surface and
\begin{itemize}
\item $S$ of type\footnote{An elliptic surface $S \rightarrow \PP^1$ is of type $E(n)$ when it has a section, $12n$ 1-nodal singular fibres, and no further singular fibres. We denote the class of its fibre by $F$.} $E(3)$ and $\vd \leq 6$,
\item $S$ of type $E(4)$ and $\vd \leq 4$,
\item $S$ of type $E(5)$ and $\vd \leq 2$,
\item $S$ of type $E(3)$, $c_1$ satisfies $c_1F \equiv 1,2 \mod 3$, $y=1$, and $\vd \leq 18$, 
\item $S$ of type $E(4)$ or $E(5)$, $y=1$, and $\vd \leq 16$,
\end{itemize}
\item $S$ is the blow up of an elliptic surface of type $E(3)$ in a point, $y=1$, and $\vd \leq 20$,
\item $S$ is a double cover of $\PP^2$ branched along a smooth octic, $y=1$, and $\vd \leq 4$,
\item[$^{\star}(7)$] $S$ a double cover of $\PP^2$ branched along a smooth octic and
\begin{itemize}
\item$\vd \leq 4$, 
\item $y=1$ and $\vd \leq 12$,
\end{itemize}
\item[(8)] $S$ is the blow up of a double cover of $\PP^2$ branched along a smooth octic, $y=1$, and $\vd \leq 4$,
\item[$^{\star}(9)$] $S$ is the blow up of a double cover of $\PP^2$ branched along a smooth octic and 
\begin{itemize}
\item $\vd \leq 4$,
\item $y=1$ and $\vd \leq 8$,
\end{itemize}
\item[$^{\star}(10)$] $S$ is a double cover of $\PP^1 \times \PP^1$ branched along a smooth curve of bidegree $(6,6)$, $y=1$, and $\vd \leq 6$,
\item[(11)] $S$ is the blow-up of a double cover of $\PP^1 \times \PP^1$ branched along a smooth curve of bidegree $(6,6)$, $y=1$, and $\vd \leq 6$,
\item[$^{\star}(12)$] $S$ is the blow-up of a double cover of $\PP^1 \times \PP^1$ branched along a smooth curve of bidegree $(6,6)$, $y=1$, and $\vd \leq 8$,
\item[$^{\star}(13)$] $S$ is a smooth quintic surface in $\PP^3$, $y=1$, and $\vd \leq 4$,
\item[$^{\star}(14)$] $S$ is the blow up of a smooth quintic surface in $\PP^3$, $y=1$, and $\vd \leq 6$.
\end{enumerate}

For $y=1$, this list contains several verifications for surfaces satisfying $K^2>0$. For general $y$, the only verification for surfaces satisfying $K^2>0$ are (7) and (9). In order to get further evidence, we turn our attention to a numerical version of Conjecture \ref{conj1}.

Suppose $S$ is a surface satisfying $b_1=0$, $p_g>0$, and its only Seiberg-Witten basic classes are $0$ and $K \neq 0$. Then $\SW(0) = 1$, $\SW(K) = (-1)^{\chi}$ and the formula of Conjecture \ref{conj1} only depends on 
$$
(\beta_1,\beta_2,\beta_3,\beta_4):=(c_1^2, c_1 K, K^2, \chi).
$$
So for each $v:=\vd$, the formula for $\overline{\chi}_{-y}^\vir(M)$ of Conjecture \ref{conj1} gives an explicit universal function
\begin{equation} \label{F}
(\beta_1,\beta_2,\beta_3,\beta_4) \mapsto F_v(\beta_1,\beta_2,\beta_3,\beta_4,y) \in \Q[y^{\pm\frac{1}{2}}].
\end{equation}
Assume the strong form of Mochizuki's formula (Remark \ref{assumpmocthm}) and $S$ is a surface as above. Multiplying by $y^{-\frac{v}{2}}$, expression \eqref{keyeqn} is also a universal function in $v$ and $(\beta_1,\beta_2,\beta_3,\beta_4)$, which we denote by
\begin{equation} \label{G}
(\beta_1,\beta_2,\beta_3,\beta_4) \mapsto G_v(\beta_1,\beta_2,\beta_3,\beta_4,y)  \in \Q[y^{\pm \frac{1}{2}}].
\end{equation}
This follows from Proposition \ref{univ}. Turning away from geometric examples, one can wonder whether \eqref{F} equals \eqref{G} for all $(\beta_1,\beta_2,\beta_3,\beta_4) \in \Z^4$ and $v \geq 0$. This turns out to be false. However, we conjecture the following:
\begin{conjecture} \label{numconj}
For all $(\beta_1,\beta_2,\beta_3,\beta_4) \in \Z^4$ and $v \geq 0$ satisfying
\begin{align}
\begin{split} \label{conditions}
\beta_1 &\equiv \beta_2 \mod 2, \\
0 &\leq v \leq \beta_1 - 3\beta_2 +4 \beta_4, \\
\beta_3 &\geq \beta_4 - 3, \quad \beta_3 \geq -1,
\end{split}
\end{align}
we have $F_v(\beta_1,\beta_2,\beta_3,\beta_4,y) = G_v(\beta_1,\beta_2,\beta_3,\beta_4,y)$.
\end{conjecture}
The first equality of \eqref{conditions} corresponds to $c_1^2 \equiv c_1 K \mod 2$. The second inequality corresponds to the essential condition (ii) in Mochizuki's theorem. The last two inequalities were found to be necessary from computer experiments. 

\begin{remark}
Conjecture \ref{numconj} and the strong form of Mochizuki's fomula (Remark \ref{assumpmocthm}) imply Conjecture \ref{conj1} for surfaces $S$ satisfying $b_1=0$, $p_g>0$, and whose only SW basic classes are $0$, $K\neq0$. The rank 2 analog of this statement was proved in \cite[Prop.~6.3]{GK1}. The same proof applies to the rank 3 case.
\end{remark}

We checked Conjecture \ref{numconj} in the following cases, for many values of $\beta_1, \beta_2$:
\begin{enumerate}
\item $\beta_3=1$, $\beta_4=0$, and $v \leq 6$,
\item $\beta_3=1$, $\beta_4=1$, and $v \leq 6$,
\item $\beta_3=1$, $\beta_4=2$, and $v \leq 4$,
\item $\beta_3=1$, $\beta_4=3$, and $v \leq 4$,
\item $\beta_3=2$, $\beta_4=0$, and $v \leq 2$,
\item $\beta_3=2$, $\beta_4=1$, and $v \leq 4$,
\item $\beta_3=2$, $\beta_4=2$, and $v \leq 4$,
\item $\beta_3=2$, $\beta_4=3$, and $v \leq 4$.
\end{enumerate}
Cases (1) and (5) do not correspond to smooth projective surfaces (their Euler characteristic is negative, yet $K^2 > 0$). For cases (2) and (6), there (obviously) are no smooth projective surfaces with $b_1=0$ and $p_g>0$. Interestingly, there are minimal surfaces of general type satisfying $b_1=0$ and
\begin{itemize}
\item $p_g=1, K^2=1$ (case (3)) by Kanev \cite{Kyn},
\item $p_g=1, K^2=2$ (case (7)) by Catanese-Debarre \cite{CD}, 
\item $p_g=2, K^2=1,2$ (cases (4), (8)) by Persson \cite[Prop.~3.23]{Per}.
\end{itemize}

\section{Monopole branch} \label{monsec}

In \cite{MT}, Maulik-Thomas introduce $y$-refined $\SU(r)$ Vafa-Witten invariants $$\VW_S^H(r,c_1,c_2,y).$$ See also \cite{Tho}. As mentioned in the introduction, on the instanton branch the definition reduces to virtual $\chi_y$-genus. On the monopole branch the definition is more subtle. 
The evidence we present for Conj.~\ref{conj2} and Rem.~\ref{rank2conj} in this section comes entirely from calculations by Laarakker \cite{Laa} and Thomas \cite{Tho}.

\begin{remark}
We initially found Conjecture \ref{conj2} as follows. First we obtained an unrefined version of Conjecture \ref{conj1} using computer experiments and Mochizuki's formula as described in Section \ref{chiysec}. The modularity transformation \eqref{modunref} from the physics literature \cite{VW, LL} swaps (part of) the instanton contribution with the monopole contribution. Together with the formula of Conjecture \ref{conj1}, this gives a natural guess for the monopole contribution. We learned this trick from Dijkgraaf-Park-Schroers \cite{DPS} who used it in the rank 2 case in order to find the instanton formula from the monopole formula. Finally, we made a $y$-refinement of each step.
\end{remark}

\subsection*{Smooth moduli spaces}

Suppose $(S,H)$ is a polarized surface such that $b_1=0$, $|K|$ contains a smooth connected canonical curve, and any line bundle $L$ on $S$ satisfying $0 \leq \deg L \leq \frac{1}{2} \deg K$ is trivial.

The monopole branch of $N_S^H(2,c_1,c_2)^{\C^*}$ is smooth if and only if $c_2 \leq 3$ \cite{TT1}. Similarly, the monopole branch of $N_S^H(3,c_1,c_2)^{\C^*}$ is smooth if and only if $c_2 \leq 2$ \cite{Laa}. In these cases, the monopole contribution to the Vafa-Witten invariants can be calculated directly by intersection theory on the moduli space.
\begin{theorem}[Thomas]
Let $S$ be as above. For $c_2=0,1,2$, the contribution of the monopole branch to $\VW_S^H(2,K,c_2,y)$ is given by Remark \ref{rank2conj}.
\end{theorem}

\begin{theorem}[Laarakker]
Let $S$ be as above. For $c_2=0,1,2$, the contribution of the monopole branch to $\VW_S^H(3,K,c_2,y)$ is given by Conjecture \ref{conj2}.
\end{theorem}

For completeness we list the predictions of Remark \ref{rank2conj} and Conjecture \ref{conj2} for $c_1=K$ for ranks 2 and 3 respectively:
{\scriptsize{
\begin{align*}
&q^{-\chi - \frac{1}{6} K^2} \cdot (y^{\frac{1}{2}} + y^{-\frac{1}{2}})^{-\chi-K^2} \\
&\times \Bigg\{ 1-2K^2 q+\Big((-y+2K^2-2-y^{-1}) K^2 + (y^2 + 10 + y^{-2}) \chi \Big) q^2 + \cdots \Bigg\},
\end{align*}
\begin{align*}
&(-1)^{\chi} q^{-\frac{3}{2}\chi - \frac{5}{24} K^2} (y + 1 + y^{-1})^{-\chi-K^2} \Bigg\{ 1-(y+2+y^{-1})K^2 q  \\
&+\frac{1}{2} \Big(  (K^2-1)y^2 + (4K^2-2)y + 6K^2-6 + (4K^2-2)y^{-1} + (K^2-1) y^{-2}   \Big) K^2 q^2 + \cdots \Bigg\}.
\end{align*}
}}
The terms $q^0,q^1,q^2$ correspond to monopole components of the moduli space for $c_2 = 0,1,2$.

\subsection*{Monopole universality} 

Let $(S,H)$ be a polarized surface with $H_1(S,\Z) = 0$ and $p_g>0$. Suppose $r,c_1$ are chosen such that there are no rank $r$ strictly Gieseker $H$-semistable Higgs pairs on $S$ with first Chern class $c_1$. A Higgs pair $(E,\phi)$ on the monopole branch $N_S^H(r,c_1,c_2)^{\C^*}$ decomposes into eigensheaves with respect to the $\C^*$ action
$$
E = \bigoplus_{i} E_i,
$$
where finitely many $E_i \neq 0$. Higgs pairs with different sequences of ranks $\{r_i\}_{i}$ occur in different connected components of $N_S^H(r,c_1,c_2)^{\C^*}$. Denote the contribution of Higgs pairs with ranks $(1, \ldots, 1)$ to $\sfZ_{S,H,r,c_1}^{\mono}(q,y)$ by
$$
\sfZ_{S,H,r,c_1}^{(1^r)}(q,y).
$$
In \cite{GT1, GT2}, Gholampour and Thomas express this contribution in terms of virtual cycles on nested Hilbert schemes of curves and points on $S$ (see also \cite{GSY}). This leads to an expression in terms of Seiberg-Witten invariants of $S$ and intersection numbers on $S^{[n_1]} \times \cdots \times S^{[n_r]}$. Based on this result, Laarakker shows the following \cite{Laa}:
\begin{theorem}[Laarakker]
For any $r>1$, there exist universal Laurent series
$$
A^{(r)}(q,y), \qquad q^{-\frac{r}{24}}B^{(r)}(q,y), \qquad \{C_{ij}^{(r)}(q,y) \}_{1 \leq i \leq j \leq r-1}, 
$$
in $\Q(y^{\frac{1}{2}})(\!(q^{\frac{1}{2r}})\!)$ with the following property. Let $S$ be any smooth projective surface with polarization $H$ and satisfying $p_g>0$ and $H_1(S,\Z) = 0$. Let $H,r,c_1$ be chosen such that there exist no rank $r$ strictly Gieseker $H$-semistable Higgs pairs on $S$ with first Chern class $c_1$. Then 
\begin{align*}
\sfZ_{S,H,r,c_1}^{(1^r)}(q,y) = &\Big( A^{(r)} \Big)^{\chi}  \Big( B^{(r)} \Big)^{K^2} \!\!\!\!\! \!\! \sum_{a_1, \ldots, a_{r-1} \in H^2(S,\Z)} \!\! \delta_{c_1,\sum_{i=1}^{r-1} i a_i} \prod_{i=1}^{r-1} \SW(a_i) \prod_{i \leq  j} \Big( C_{ij}^{(r)} \Big)^{a_i a_j},
\end{align*}
where $\delta_{a,b}$ was defined in \eqref{deltadef}.
\end{theorem}

Similar to Section \ref{chiysec}, Laarakker shows that these universal functions are determined on $\PP^2$ and $\PP^1 \times \PP^1$. Using torus localization, he calculates these universal functions up to some order. Normalizing such that the RHS is a formal power series in $q$ starting with constant coefficient 1, he obtains \cite{Laa}:
\begin{align*}
-(y^{\frac{1}{2}} + y^{-\frac{1}{2}}) q A^{(2)} &=  \frac{(y-y^{-1}) q}{\phi_{-2,1}(q^2,y^2)^{\frac{1}{2}} \Delta(q^2)^{\frac{1}{2}}}   \mod q^7 \\
q^{-\frac{1}{12}} B^{(2)} &=\frac{q^{-\frac{1}{12}} \eta(q)^2}{ \theta_3(q,y)} \mod q^7 \\
(y^{\frac{1}{2}} + y^{-\frac{1}{2}}) q^{\frac{1}{4}} C^{(2)}_{11} &=\frac{(y^{\frac{1}{2}} + y^{-\frac{1}{2}}) q^{\frac{1}{4}} \theta_3(q,y)}{\theta_2(q,y)} \mod q^7 \\
(y+1+y^{-1}) q^{\frac{3}{2}} A^{(3)} &= \frac{(y^{\frac{3}{2}} - y^{-\frac{3}{2}}) q^{\frac{3}{2}}}{\phi_{-2,1}(q^3,y^3)^{\frac{1}{2}} \Delta(q^3)^{\frac{1}{2}}}  \mod q^6 \\
q^{-\frac{1}{8}} B^{(3)} &= \frac{q^{-\frac{1}{8}} \eta(q)^3 W_{-}(q^{\frac{1}{2}},y)}{\Theta_{A_2,(1,0)}(q^{\frac{1}{2}},y)}   \mod q^6 \\
(y+1+y^{-1}) q^{\frac{1}{3}} C^{(3)}_{11} &= \frac{(y+1+y^{-1}) q^{\frac{1}{3}}}{W_{-}(q^{\frac{1}{2}},y)} \mod q^6\\
(y+1+y^{-1}) q^{\frac{1}{3}} C^{(3)}_{22} &=  \frac{(y+1+y^{-1}) q^{\frac{1}{3}}}{W_{-}(q^{\frac{1}{2}},y)} \mod q^6 \\
\frac{y+1+y^{-1}}{y+2+y^{-1}} q^{\frac{1}{3}} C^{(3)}_{12} &= \frac{y+1+y^{-1}}{y+2+y^{-1}} q^{\frac{1}{3}} W_{+}(q^{\frac{1}{2}},y) W_{-}(q^{\frac{1}{2}},y) \mod q^6.
\end{align*}
This precisely recovers Remark \ref{rank2conj} and Conjecture \ref{conj2} up to the given orders.

Besides providing evidence for our conjectures, Laarakker's calculations suggest the following 
\begin{equation} \label{1^rconj}
\sfZ_{S,H,r,c_1}^{(1^r)}(q,y) \stackrel{?}{=} \sfZ_{S,H,r,c_1}^{\mono}(q,y).
\end{equation}
For rank 2 this is obvious and for rank 3 it implies that Higgs pairs with ranks $(1,2)$ and $(2,1)$ do not contribute. Indeed, for low prime rank (such as $r=3$) or $S = K3$ and any prime rank, Thomas establishes \eqref{1^rconj} using an interesting cosection argument \cite{Tho}.

\section{Modularity} \label{modsec}

In this section, we give evidence for Conjecture \ref{conj3}. We show, among other things, that that our conjectural formulae for $\sfZ_{S,H,2,c_1}(q,y)$ and $\sfZ_{S,H,3,c_1}(q,y)$ (Conjecture \ref{conj2} and Remark \ref{rank2conj}) satisfy the $y$-refined modularity transformation of Conjecture \ref{conj3}. This involves a delicate interplay between quite diverse mathematical objects:
\begin{itemize}
\item properties of Seiberg-Witten invariants,
\item lattice theory of $(H^2(S,\Z),\cup)$,
\item Gauss sums and Dedekind sums,
\item transformations of theta functions.
\end{itemize}

\subsection{Seiberg-Witten invariants}

Let $S$ be a smooth projective surface satisfying $p_g>0$ and $b_1=0$. Then any Seiberg-Witten basic class $a \in H^2(S,\Z)$ satisfies (\cite[Sect.~6.3]{Moc} or \cite{Mor})
\begin{align} \label{SWeqn}
a K = a^2, \qquad \SW(K - a) = (-1)^{\chi} \, \SW(a).
\end{align}

\subsection{Lattice sums and Gauss sums}

Let $S$ be a smooth projective surface satisfying $H_1(S,\Z) = 0$ and $p_g>0$. Then $H^2(S,\Z)$ is torsion free and we consider the unimodular lattice $(H^2(S,\Z),\cup)$. For any prime $p$, we have
$$
H^2(S,\Z) / pH^2(S,\Z) \cong H^2(S,\Z) \otimes \Z_p \cong H^2(S,\Z_p) 
$$
with its induced pairing. 
We denote the Betti numbers of $S$ by $b_i$ and its signature by $\sigma$. In particular, $b_2 = b_2^+ + b_2^-$ and $\sigma = b_2^+ - b_2^-$. Define
$$
\delta_{a,b}^{(p)} := \left\{\begin{array}{cc} 1 & \textrm{if \ } a-b \in pH^2(S,\Z)  \\ 0 & \textrm{otherwise}.  \end{array}\right.
$$
We usually write $\delta_{a,b} =\delta_{a,b}^{(p)}$, when $p$ is fixed. The following results are due to Vafa-Witten and Labastida-Lozano \cite{VW, LL}. 
\begin{proposition}[Vafa-Witten] \label{VWlattice}
\begin{align*}
\sum_{[x] \in H^2(S,\Z_2)} (-1)^{c_1x}&= 2^{b_2} \delta_{c_1,0}, \\
\sum_{[x] \in H^2(S,\Z_2)} (-1)^{c_1x} i^{x^2}&= 2^{\frac{b_2}{2}} i^{\frac{\sigma}{2} - c_1^2}.
\end{align*}
\end{proposition}
\begin{proposition}[Labastida-Lozano] \label{LLlattice}
For $r>2$ prime and $m =1, \ldots, r-1$, we have
\begin{align*}
\sum_{[x] \in H^2(S,\Z_r)} e^{\frac{2 \pi i}{r} (c_1x)} &= r^{b_2} \delta_{c_1,0}, \\
\sum_{[x] \in H^2(S,\Z_r)} e^{\frac{2 \pi i}{r} (c_1x)} e^{\frac{\pi i (r-1)}{r} m x^2} &= \epsilon(m)^{b_2} r^{\frac{b_2}{2}} e^{-\frac{\pi i}{8} (r-1)^2 \sigma} e^{\frac{\pi i (r-1)}{r} n c_1^2},
\end{align*}
where $mn \equiv -1 \mod r$,
$$
\epsilon(m) := \left\{ \begin{array}{cc} \Big( \frac{m/2}{r} \Big) & \textrm{if} \ m \ \textrm{is even} \\ \Big( \frac{(m+r)/2}{r} \Big) & \textrm{if} \ m \ \textrm{is odd} \end{array} \right.
$$
and $\big( \frac{a}{r} \big)$ denotes the Legendre symbol.
\end{proposition}
Since $(H^2(S,\Z), \cup)$ is unimodular of rank $b_2$, the first equation in both propositions is clear. As described by Vafa-Witten and Labastida-Lozano, the other two identities can proved using basic facts from lattice theory and Gauss sums. For the sake of completeness, we include the argument. 

For any $r \geq 2$ prime, $m=1, \ldots, r-1$, and unimodular lattice $L$, define
\begin{align*}
\phi(L) &:= r^{-\rk(L)/2} \sum_{[x] \in L / r L} e^{\frac{\pi i (r-1)}{r} m x^2}.
\end{align*}
Then
\begin{align} \label{product}
\phi(L_1 \oplus L_2) &= \phi(L_1) \phi(L_2), 
\end{align}
for all unimodular lattices $L_1$, $L_2$. Let $I_{\pm}$ denote the rank 1 lattice $\Z$ with quadratic form $\pm x^2$. Any odd indefinite unimodular lattice is of the form $m I_+ \oplus nI_-$ for some $m,n \geq 1$ \cite[Ch.~V.2.2]{Ser}. We can make $H^2(S,\Z)$ odd and indefinite after adding $I_{+}$ or $I_{-}$, so \eqref{product} implies
\begin{equation} \label{prodIpm}
\phi(H^2(S,\Z)) = \phi(I_+)^{b_2^+} \phi(I_-)^{b_2^-} = \phi(I_+)^{b_2^+} \overline{\phi(I_+)}^{b_2^-}.
\end{equation}
This argument requires $\phi(I_{\pm}) \neq 0$, which we now show by calculation. 

For $r=2$, we have 
$$
\phi(I_+) = \frac{1+i}{\sqrt{2}}
$$
and the second equation of Proposition \ref{VWlattice}, for $c_1 = 0$, follows from \eqref{prodIpm}. The formula for $c_1 \neq 0$ follows by replacing the sum over $x$ by a sum over $x+c_1$.

For $r \geq 3$ prime and $m=1, \ldots, r-1$, we want to calculate
$$
\sqrt{r} \phi(I_+) = \Bigg\{ \begin{array}{cc} \overline{G(m/2,r)} & \textrm{if} \ m \ \textrm{even} \\ \overline{G((m+r)/2,r)} & \textrm{if} \ m \ \textrm{odd}, \end{array} 
$$
where
$$
G(m,r) := \sum_{x=0}^{r-1} e^{\frac{2 \pi i}{r} m x^2}.
$$ 
This is a familiar object from number theory known as a Gauss sum. The second equation of Proposition \ref{LLlattice} for $c_1=0$ follows from \cite[Ch.~4.3]{Lan}, which states (after some rewriting)
$$
G(m,r) = \Big( \frac{m}{r} \Big) \sqrt{r} e^{\frac{\pi i}{8} (r-1)^2},  
$$
for any odd number $r > 0$ and $m \in \Z$ such that $\gcd(r,m) = 1$. The formula for $c_1 \neq 0$ follows by replacing the sum over $x$ by a sum over $x+n c_1$, where $mn \equiv -1 \mod r$.

\subsection{Dedekind sums} 

We often encounter the expression
$$
\phi_{-2,1}(q,y)^{\frac{1}{2}} \Delta(q)^{\frac{1}{2}} := (y^{\frac{1}{2}} - y^{-\frac{1}{2}}) \, q^{\frac{1}{2}} \, \prod_{n=1}^{\infty} (1-q^n)^{10} (1-q^n y) (1-q^n y^{-1}),
$$
where $\phi_{-2,1}(q,y) \Delta(q)$ is the unique Jacobi cusp form of weight 10 and index 1 \cite{EZ}. It transforms as follows
\begin{equation} \label{transepsilon}
\phi_{-2,1}\Big(\frac{a \tau + b}{c \tau + d},\frac{z}{c \tau + d}\Big)^{\frac{1}{2}} \Delta\Big(\frac{a \tau + b}{c \tau + d}\Big)^{\frac{1}{2}} = \epsilon(a,b,c,d) (c \tau + d)^{5} e^{\frac{\pi i c z^2}{c \tau + d}} \phi_{-2,1}(\tau,z)^{\frac{1}{2}} \Delta(\tau)^{\frac{1}{2}}
\end{equation}
for all 
$$
\Big( \begin{array}{cc} a & b \\ c & d \end{array}\Big) \in \mathrm{SL}(2,\Z),
$$
where $\epsilon(a,b,c,d)^2=1$. Clearly $\epsilon(a,b,c,d) = e^{\pi i b}$ when $c=0$ and $a=d=1$. The following lemma determines the signs $\epsilon(a,b,c,d)$ for $c>0$.
\begin{lemma} \label{Dedekind} For $c>0$ we have
$$
\epsilon(a,b,c,d) = - e^{\pi i \Big( \frac{a+d}{c} + 12 s(-d,c) \Big)},
$$
where $s(h,k)$ denotes the Dedekind sum
$$
s(h,k) = \sum_{r=1}^{k-1} \frac{r}{k} \Big(\frac{hr}{k} - \Big\lfloor \frac{hr}{k} \Big\rfloor - \frac{1}{2} \Big).
$$
\end{lemma}
\begin{proof}
Using 
$$
\frac{e^{\pi i z} - e^{-\pi i z}}{e^{\frac{\pi i z}{c \tau + d}} - e^{-\frac{\pi i z}{c \tau + d}}} = c \tau + d + O(z),
$$
and setting $z=0$, \eqref{transepsilon} becomes
\begin{align*}
\eta\Big( \frac{a \tau + b}{c \tau + d} \Big)^{12} &= \epsilon(a,b,c,d) (c \tau + d)^{6} \eta(\tau)^{12} \\
&= - e^{\pi i \Big( \frac{a+d}{c} + 12 s(-d,c) \Big)} \eta(\tau)^{12},
\end{align*}
where the second equality follows from the transformation laws of the Dedekind eta function \cite[Thm.~3.4]{Apo}.
\end{proof}

\subsection{Theta functions}

In this section, we review some facts about theta functions (e.g.~see \cite[Sect.~3.3]{GZ}).
Let $\Gamma$ be a positive definite lattice of rank $r$.  We write 
$V=\Gamma\otimes \C$ and $V_\Q=\Gamma\otimes\Q$.  
For vectors $v,w\in V$, let $\langle v,w\rangle$
be the bilinear form.
Denote by $M_\Gamma$ the set of meromorphic functions $f:{\mathfrak H}\times V\to \C$. 
For 
$(\lambda,\mu)\in V \times V$, 
let
\begin{align*}
f|(\lambda,\mu)(\tau,x)&=q^{\frac{1}{2} \langle \lambda,\lambda\rangle} \exp(2\pi i \langle \lambda,x+\mu/2\rangle) \,  f(\tau,x+\lambda\tau+\mu).
\end{align*}
We also write $$f|_{r/2}S(\tau,x):= \Big( \frac{\tau}{i} \Big)^{-\frac{r}{2}} \, e^{-\pi i \langle x,x\rangle/\tau}f(-1/\tau,x/\tau).$$
It is easy to see that 
\begin{align}\label{Sv} f|(\lambda,\mu)|_{r/2}S(\tau,x)= f|_{r/2}S |(\mu,-\lambda)(\tau,x).
\end{align}
The theta function for $\Gamma$ is 
$$\Theta_\Gamma(\tau,x) := \sum_{v\in \Gamma} q^{\frac{1}{2}\langle v, v\rangle} e^{2 \pi i \langle v ,x\rangle}\in M_{\Gamma}.$$ 
If $\Gamma$ has rank $r$ it is well-known that
\begin{align} \label{STh}\sqrt{N} \,
\Theta_{\Gamma}|_{r/2}S(\tau,x)
=\Theta_{\Gamma^\vee}(\tau,x)=\sum_{v\in P} \Theta_{\Gamma}|(v,0)(\tau,x).
\end{align}
Here is $N$ the determinant of the matrix of the bilinear form on $\Gamma$ and
$$\Gamma^\vee:=\big\{v\in \Gamma\otimes \Q\bigm| \langle v,w\rangle\in \Z,  \forall \, w\in \Gamma \big\}$$
is the dual lattice to $\Gamma$. For the second equality of \eqref{STh}, we assume that $\Gamma$ is integral and $P$ is a system of representatives of $\Gamma^\vee/\Gamma$.

The $A_1$ lattice consists of $\Z$ with bilinear form with ``matrix'' $(2)$. 
Then the theta functions appearing in the rank 2 conjectures of Section \ref{intro} can be expressed in terms of $\Theta_{A_1}(\tau,x)$ and $\Theta_{A_1^\vee}(\tau,x)$ as follows 
\begin{align*}
\theta_3(q,y) = \Theta_{A_1}(\tau,z/2), \quad \theta_3(q,y) + \theta_2(q,y) = \Theta_{A_1^\vee}(\tau,z/2).
\end{align*}

Next, consider the $A_2$ lattice and let $A$ denote the matrix corresponding to its bilinear form (see \eqref{matrixA}). We also consider the dual lattice $A_2^\vee$. Taking 
\begin{equation} \label{basis}
(2/3,1/3),(1/3,2/3)
\end{equation}
as its basis, the matrix corresponding to its bilinear form is $A^{-1}$ (see \eqref{matrixAinv}). Then the theta functions appearing in the rank 3 conjectures of Section \ref{intro} can be expressed in terms of $\Theta_{A_2}(\tau,x)$ and $\Theta_{A_2^\vee}(\tau,x)$ as follows
\begin{align*}
\Theta_{A_2,(0,0)}(q^{\frac{1}{2}},y)&=\Theta_{A_2}(\tau,(z,z)),\quad \Theta_{A_2,(1,0)}(q^{\frac{1}{2}},y)=
\Theta_{A_2}|(1/3,-1/3),0)(\tau,(z,z)), \\
\Theta_{A^\vee_2,(0,0)}(q^{\frac{1}{6}},y)&=\Theta_{A^\vee_2}(\tau,(z,z)),\quad \Theta_{A^\vee_2,(0,1)}(q^{\frac{1}{6}},y)=\Theta_{A^\vee_2}|(0,(1,-1))(\tau,(z,z)).
\end{align*}

\subsection{Rank 1 and specialization} 

\begin{proposition}
Conjecture \ref{conj3} holds for $r=1$.
\end{proposition}
\begin{proof}
Let $S$ be a smooth projective surface satisfying $b_1=0$ ($p_g>0$, $H_1(S,\Z) = 0$ are not needed in this proof). When $r=1$, we have
\begin{align*}
\sfZ_{S,H,1,c_1}(q,y) &= q^{-\frac{1}{2} \chi + \frac{1}{24} K^2} \sum_{n=0}^{\infty} \overline{\chi}_{-y}(\Hilb^n(S)) \, q^{n} \\
&= \Bigg( \frac{y^{\frac{1}{2}} - y^{-\frac{1}{2}}}{\phi_{-2,1}(q,y)^{\frac{1}{2}} \Delta(q)^{\frac{1}{2}}} \Bigg)^{\chi} \Bigg( \frac{1}{\eta(q)} \Bigg)^{-K^2}
\end{align*}
by a result of the first named author and W.~Soergel \cite{GS}. \\

\noindent \textbf{Step 1:} For $\tau \mapsto \tau+1$, $z \mapsto z$, the result follows from the following transformations
\begin{align*}
\eta(q) &\mapsto e^{\frac{\pi i}{12}} \eta(q), \\
\phi_{-2,1}(q,y)^{\frac{1}{2}} \Delta(q^{\frac{1}{2}})^{\frac{1}{2}} &\mapsto - \phi_{-2,1}(q,y)^{\frac{1}{2}} \Delta(q^{\frac{1}{2}})^{\frac{1}{2}}.
\end{align*}

\noindent \textbf{Step 2:} For $\tau \mapsto -1/\tau$, $z \mapsto z/\tau$, the result follows from the following transformations
\begin{align*}
\eta(q) &\mapsto \Big( \frac{\tau}{i} \Big)^{\frac{1}{2}} \eta(q), \\
\phi_{-2,1}(q,y)^{\frac{1}{2}} \Delta(q^{\frac{1}{2}})^{\frac{1}{2}} &\mapsto - \tau^5 e^{\frac{\pi i z^2}{\tau}} \phi_{-2,1}(q,y)^{\frac{1}{2}} \Delta(q^{\frac{1}{2}})^{\frac{1}{2}},
\end{align*}
where the minus sign in the second equation comes from Lemma \ref{Dedekind}. 
\end{proof}

\begin{proposition}
Conjecture \ref{conj3} implies transformations \eqref{modunref}.
\end{proposition}
\begin{proof}
Taking the limit $z \rightarrow 0$, \eqref{modref} implies \eqref{modunref}. This follows from
$$
\frac{\Big(y^{\frac{1}{2}} - y^{-\frac{1}{2}}\Big)\Big|_{\frac{z}{\tau}}}{y^{\frac{1}{2}} - y^{-\frac{1}{2}}} = \frac{1}{\tau} + O(z)
$$
and a simple calculation involving Noether's formula 
\begin{equation*}
\chi = \frac{1}{12}(K^2+e). \qedhere
\end{equation*}
\end{proof}

\subsection{$K3$ surfaces}

In \cite{LL} (and \cite{VW} when $c_1=0$), the authors conjecture a formula for $\sfZ_{K3,H,r,c_1}(q)$ for $r$ prime. When $c_1=0$, this formula was proved by Tanaka-Thomas \cite{TT2} (and extended to all integers $r>0$). When $r$ does not divide $Hc_1$, $M_S^H(r,c_1,c_2)$ is deformation equivalent to $S^{[\vd/2]}$ \cite{Yos4} and the formula essentially follows from \cite{GS}. We conjecture a natural $y$-refinement:
\begin{conjecture} \label{K3conj}
For any K3 surface with polarization $H$, first Chern class $c_1$, and $r$ prime, we have
\begin{align*}
\sfZ_{K3,H,r,c_1}^{\inst}(q,y) &= \frac{1}{r} \sum_{m=0}^{r-1} e^{\frac{i \pi (r-1)}{r} m c_1^2} \frac{(y^{\frac{1}{2}} - y^{-\frac{1}{2}})^2}{\phi_{-2,1}\Big(\frac{\tau+m}{r},z\Big) \Delta\Big(\frac{\tau+m}{r}\Big)}, \\
\sfZ_{K3,H,r,c_1}^{\mono}(q,y) &= \frac{(y^{\frac{1}{2}} - y^{-\frac{1}{2}})^2}{\phi_{-2,1}(r\tau,rz) \Delta(r \tau)} \delta_{c_1,0}.
\end{align*}
\end{conjecture}

The following can be seen as evidence for Conjecture \ref{conj3}. The proof uses the well-known transformation properties of $\phi_{-2,1}$ and $\Delta$ \cite{EZ}.
\begin{proposition}
Conjecture \ref{K3conj} implies Conjecture \ref{conj3} for K3 surfaces.
\end{proposition}
\begin{proof}
\noindent \textbf{Step 1:} For $\tau \mapsto \tau+1$, $z \mapsto z$, we claim
$$
\sfZ_{K3,H,r,c_1}(q,y) \mapsto e^{-\frac{i \pi(r-1)}{r} c_1^2} \sfZ_{K3,H,r,c_1}(q,y).
$$
The monopole branch contribution does not transform. For the transformation of the instanton contribution we use
\begin{align*}
 \Delta\Big(\frac{\tau+m}{r}\Big) &\mapsto \Delta\Big(\frac{\tau+m+1}{r}\Big), \\
\phi_{-2,1}\Big(\frac{\tau+m}{r},z\Big) &\mapsto \phi_{-2,1}\Big(\frac{\tau+m+1}{r},z\Big).
\end{align*}
This transformation maps term $m$ to term $m+1$ (modulo $r$) up to a factor $e^{\frac{\pi i(r-1)}{r} c_1^2}$. In particular, term $r-1$ gets mapped to term $0$, for which we use
$$
e^{\pi i (r-1) c_1^2} = 1.
$$
This equality follows at once from the fact that $c_1^2$ is even on a $K3$ surface. \\

\noindent \textbf{Step 2:} For $\tau \mapsto -1/\tau$, $z \mapsto z/\tau$, we claim
$$
\frac{\sfZ_{K3,H,r,c_1}(q,y)}{(y^{\frac{1}{2}} - y^{-\frac{1}{2}})^2} \mapsto r^{-11} \, \tau^{-10} \, e^{-\frac{2 \pi i z^2}{\tau} r} \sum_{[a]} e^{\frac{2 \pi i}{r} c_1a} \frac{\sfZ_{K3,H,r,a}(q,y)}{(y^{\frac{1}{2}} - y^{-\frac{1}{2}})^2}.
$$
For any $m = 1, \ldots, r-1$, we define $n \in \{1, \ldots, r-1\}$ by the equation $mn \equiv -1 \mod r$. In order to prove the claim, we use the following transformations
\begin{align*}
\Delta(r \tau) &\mapsto \Big( \frac{\tau}{r} \Big)^{12} \Delta\Big( \frac{\tau}{r} \Big), \\
\Delta\Big( \frac{\tau+m}{r} \Big) &\mapsto \tau^{12} \Delta\Big( \frac{\tau+n}{r} \Big), \\
\phi_{-2,1}(r\tau,rz) &\mapsto \Big( \frac{\tau}{r} \Big)^{-2} \, e^{\frac{2 \pi i z^2}{\tau} r} \phi_{-2,1}\Big(\frac{\tau}{r},z \Big), \\
\phi_{-2,1}\Big( \frac{\tau+m}{r},z \Big) &\mapsto \tau^{-2} \, e^{\frac{2 \pi i z^2}{\tau} r} \phi_{-2,1}\Big( \frac{\tau+n}{r},z \Big).
\end{align*}
We also use the following lattice identities (Propositions \ref{VWlattice}, \ref{LLlattice})
\begin{align*}
\sum_{[x] \in H^2(K3,\Z_r)} e^{\frac{2 \pi i}{r} (c_1x)} &= r^{22} \delta_{c_1,0}, \\
\sum_{[x] \in H^2(K3,\Z_r)} e^{\frac{2 \pi i}{r} (c_1x)} e^{\frac{\pi i (r-1)}{r} m x^2} &= r^{11} e^{\frac{\pi i (r-1)}{r} n c_1^2},
\end{align*}
where we used $(-1)^{c_1^2}=1$. The result now follows from a direct calculation.
\end{proof}

\subsection{Rank 2 modularity} 

\begin{proposition} 
The formula for $\sfZ_{S,H,2,c_1}(q,y)$ from Remarks \ref{rank2cor} and \ref{rank2conj} satisfies Conjecture \ref{conj3}.
\end{proposition}
\begin{proof}
\textbf{Step 1:} For $\tau \mapsto \tau+1$, $z \mapsto z$, we use the following transformations
\begin{align*}
\eta(q) &\mapsto e^{\frac{\pi i}{12}} \eta(q), \\
\phi_{-2,1}(q^{\frac{1}{2}},y)^{\frac{1}{2}} \Delta(q^{\frac{1}{2}})^{\frac{1}{2}} &\mapsto \phi_{-2,1}(-q^{\frac{1}{2}},y)^{\frac{1}{2}} \Delta(-q^{\frac{1}{2}})^{\frac{1}{2}}, \\
\phi_{-2,1}(-q^{\frac{1}{2}},y)^{\frac{1}{2}} \Delta(q^{\frac{1}{2}})^{\frac{1}{2}} &\mapsto -\phi_{-2,1}(q^{\frac{1}{2}},y)^{\frac{1}{2}} \Delta(-q^{\frac{1}{2}})^{\frac{1}{2}} \\
\theta_3(q,y) &\mapsto \theta_3(q,y), \\
\theta_2(q,y) &\mapsto i \theta_2(q,y). 
\end{align*}
The rest of the calculation is straight-forward. On the instanton branch, it involves changing the summation variable $a$ in 
$$
\sum_{a \in H^2(S,\Z)} \SW(a) (-1)^{c_1 a} ( \cdots )^{a K} 
$$
to $K-a$, which gives
$$
\sum_{a \in H^2(S,\Z)} \SW(K-a) (-1)^{c_1 (K-a)} ( \cdots )^{(K-a)K}. 
$$
In addition, one requires the following equations 
\begin{align}
\begin{split} \label{listofeqn}
a K = a^2, \qquad\SW(K - a) = (-1)^{\chi} \, \SW(a), \qquad(-1)^{c_1 K} = (-1)^{c_1^2},
\end{split}
\end{align}
where the first two equations hold for any Seiberg-Witten basic class $a$ by \eqref{SWeqn}. The third is Wu's formula, which holds for any $c_1 \in H^2(S,\Z)$. On the monopole branch, one only requires
\begin{align*}
i^{-aK} \delta_{c_1,a} = i^{-a^2} \delta_{c_1,a}=  i^{-c_1^2} \delta_{c_1,a},
\end{align*}
where again $a$ is a Seiberg-Witten basic class. \\

\noindent \textbf{Step 2:} For $\tau \mapsto -1/\tau$, $z \mapsto z/\tau$, we use the following transformations
\begin{align*}
\eta(q) &\mapsto \Big( \frac{\tau}{i} \Big)^{\frac{1}{2}} \eta(q), \\
\phi_{-2,1}(q^2,y^2)^{\frac{1}{2}} \Delta(q^{2})^{\frac{1}{2}} &\mapsto - \Big(\frac{\tau}{2} \Big)^{5} e^{\frac{2 \pi i z^2}{\tau}} \phi_{-2,1}(q^{\frac{1}{2}},y)^{\frac{1}{2}} \Delta(q^{\frac{1}{2}})^{\frac{1}{2}}, \\
\phi_{-2,1}(-q^{\frac{1}{2}},y)^{\frac{1}{2}} \Delta(q^{\frac{1}{2}})^{\frac{1}{2}} &\mapsto - \tau^{5} e^{\frac{2 \pi i z^2}{\tau}} \phi_{-2,1}(-q^{\frac{1}{2}},y)^{\frac{1}{2}} \Delta(q^{\frac{1}{2}})^{\frac{1}{2}}, \\
\theta_3(q,y) &\mapsto \frac{1}{\sqrt{2}} \Big(\frac{\tau}{i} \Big)^{\frac{1}{2}} e^{\frac{\pi i z^2}{2\tau}} (\theta_3(q,y)+\theta_2(q,y)), \\
\theta_2(q,y) &\mapsto \frac{1}{\sqrt{2}} \Big(\frac{\tau}{i} \Big)^{\frac{1}{2}} e^{\frac{\pi i z^2}{2\tau}} (\theta_3(q,y)-\theta_2(q,y)).
\end{align*}
For the third transformation, we use coordinates $\tilde{\tau} = \frac{\tau+1}{2}$, $\tilde{z} = z$, which allows us to apply \eqref{transepsilon} and Lemma \ref{Dedekind}. The transformations for the theta functions are standard (see \eqref{STh} or  \cite[Sect.~5]{VW}).

Using Proposition \ref{VWlattice}, changing the summation variable $a$ as in Step 1 and using equations \eqref{listofeqn}, the result follows.
\end{proof}

\subsection{Rank 3 modularity}

\begin{lemma} \label{thetalemma} Under the transformation 
$\tau\mapsto -1/\tau$, $z\mapsto z/\tau$ we have 
\begin{equation} \label{transtheta00} 
\Theta_{A_2,(0,0)}(q^{\frac{1}{2}},y) \mapsto \frac{1}{\sqrt{3}} \Big( \frac{\tau}{i} \Big)  e^{\frac{2 \pi i z^2}{\tau}} \Theta_{A_2^{\vee},(0,0)}(q^{\frac{1}{6}},y).
\end{equation}
Furthermore,  we have the following identities
\begin{equation}\label{ident}
\begin{split}
\Theta_{A_2,(1,0)}(q^{\frac{1}{2}},y) &= -\frac{1}{2} \Theta_{A_2,(0,0)}(q^{\frac{1}{2}},y) + \frac{1}{2} \Theta_{A_2^\vee,(0,0)}(q^{\frac{1}{6}},y),  \\
\Theta_{A_2^{\vee},(0,1)}(q^{\frac{1}{6}},y) &= \frac{3}{2} \Theta_{A_2,(0,0)}(q^{\frac{1}{2}},y) - \frac{1}{2} \Theta_{A_2^\vee,(0,0)}(q^{\frac{1}{6}},y),  \\
\Theta_{A_2^{\vee},(0,0)}(\epsilon q^{\frac{1}{6}},y) &= (2+\epsilon) \Theta_{A_2,(0,0)}(q^{\frac{1}{2}},y) - (1+\epsilon) \Theta_{A_2^\vee,(0,0)}(q^{\frac{1}{6}},y),  \\
\Theta_{A_2^{\vee},(0,0)}(\epsilon^2 q^{\frac{1}{6}},y) &= (1-\epsilon) \Theta_{A_2,(0,0)}(q^{\frac{1}{2}},y) + \epsilon \, \Theta_{A_2^\vee,(0,0)}(q^{\frac{1}{6}},y),  \\
\Theta_{A_2^{\vee},(0,1)}(\epsilon q^{\frac{1}{6}},y) &= \frac{1-\epsilon}{2} \Theta_{A_2,(0,0)}(q^{\frac{1}{2}},y) + \frac{1+\epsilon}{2} \Theta_{A_2^\vee,(0,0)}(q^{\frac{1}{6}},y),  \\
\Theta_{A_2^{\vee},(0,1)}(\epsilon^2 q^{\frac{1}{6}},y) &= \Big(1+\frac{\epsilon}{2}\Big)\Theta_{A_2,(0,0)}(q^{\frac{1}{2}},y) -\frac{\epsilon}{2} \Theta_{A_2^\vee,(0,0)}(q^{\frac{1}{6}},y),
\end{split}
\end{equation}
where $\epsilon:=e^{\frac{2 \pi i}{3}}$. In particular, under
$\tau\mapsto -1/\tau$, $z\mapsto z/\tau$ we have 
\begin{align*}
Z(q^{\frac{1}{6}},y) &\mapsto \frac{Z(q^{\frac{1}{6}},y)+2}{Z(q^{\frac{1}{6}},y)-1} = W(q^{\frac{1}{2}},y), \\
Z_{\pm}(q^{\frac{1}{6}},y) &\mapsto W_{\pm}(q^{\frac{1}{2}},y), \\
Z_{\pm}(\epsilon q^{\frac{1}{6}},y) &\mapsto \epsilon Z_{\pm}(\epsilon^2 q^{\frac{1}{6}},y), \\
Z_{\pm}(\epsilon^2 q^{\frac{1}{6}},y) &\mapsto \epsilon^2 Z_{\mp}(\epsilon q^{\frac{1}{6}},y).
\end{align*}
\end{lemma}
\begin{proof}
The first equality of \eqref{STh} gives
\begin{equation}\label{Strans}
\begin{split}
\sqrt{3} \Big( \frac{\tau}{i} \Big)^{-1} e^{-\frac{2\pi iz^2}{\tau}}\Theta_{A_2,(0,0)}(q^{\frac{1}{2}},y)|_{(\tau,z)=(-1/\tau,z/\tau)}&
=\Theta_{A_2^\vee,(0,0)}(q^{\frac{1}{6}},y),\\
\frac{1}{\sqrt{3}} \Big( \frac{\tau}{i} \Big)^{-1}e^{-\frac{2\pi iz^2}{\tau}}\Theta_{A_2^\vee,(0,0)}(q^{\frac{1}{6}},y)|_{(\tau,z)=(-1/\tau,z/\tau)}&=
\Theta_{A_2,(0,0)}(q^{\frac{1}{2}},y).
\end{split}
\end{equation}
Here we note that $(z,z)$ with respect to the basis \eqref{basis} is also $(z,z)$. Using the first equality of \eqref{STh} combined with \eqref{Sv} gives
\begin{align*}
\Big( \frac{\tau}{i}\Big)^{-1} e^{-\frac{2\pi iz^2}{\tau}}\Theta_{A_2,(1,0)}(q^{\frac{1}{2}},y)|_{(\tau,z)=(-1/\tau,z/\tau)} &= \Theta_{A_2}|((1/3,-1/3),0)|_{1} S (\tau,(z,z)) \\
&= \Theta_{A_2}(\tau,x)|_{1} S |(0,(-1/3,1/3))(\tau,(z,z)), \\
&=\frac{1}{\sqrt{3}} \Theta_{A_2^\vee} |(0,(-1,1))(\tau,(z,z)), \\
&=\frac{1}{\sqrt{3}} \Theta_{A_2^\vee,(0,1)}(q^{\frac{1}{6}},y).
\end{align*}
Here the third equality uses that $(-\frac{1}{3},\frac{1}{3})$ with respect to the basis \eqref{basis} equals $(-1,1)$. This shows
\begin{equation} \label{10trans}
\sqrt{3} \Big( \frac{\tau}{i} \Big)^{-1} e^{-\frac{2\pi iz^2}{\tau}}\Theta_{A_2,(1,0)}(q^{\frac{1}{2}},y)|_{(\tau,z)=(-1/\tau,z/\tau)} =\Theta_{A_2^\vee,(0,1)}(q^{\frac{1}{6}},y).
\end{equation}

Since $(0,0),(1/3,-1/3),(-1/3,1/3)$
is a system of representatives for $A_2^\vee/A_2$, the second equality of \eqref{STh} gives
\begin{align*}
&\Theta_{A_2^\vee,(0,0)}(q^{\frac{1}{6}},z)= \Theta_{A_2^\vee}(\tau,(z,z)) \\
&=\Theta_{A_2}(\tau,(z,z))+\Theta_{A_2}|((1/3,-1/3),0)(\tau,(z,z))+\Theta_{A_2}|((-1/3,1/3),0)(\tau,(z,z))\\
&=\Theta_{A_2,(0,0)}(q^{\frac{1}{2}},y)+2\Theta_{A_2,(1,0)}(q^{\frac{1}{2}},y).
\end{align*}
Therefore, we obtain
\begin{equation}\label{A10A00}
\Theta_{A_2,(1,0)}(q^{\frac{1}{2}},y)=- \frac{1}{2} \Theta_{A_2,(0,0)}(q^{\frac{1}{2}},y) + \frac{1}{2} \Theta_{A_2^\vee,(0,0)}(q^{\frac{1}{6}},y).
\end{equation}
Applying $\tau\mapsto -1/\tau$, $z\to z/\tau$ to this equation and using \eqref{Strans}, \eqref{10trans} gives
\begin{equation}\label{A01A00}
\Theta_{A^\vee_2,(0,1)}(q^{\frac{1}{6}},y)=\frac{3}{2}\Theta_{A_2,(0,0)}(q^{\frac{1}{2}},y)-\frac{1}{2}\Theta_{A_2^\vee,(0,0)}(q^{\frac{1}{6}},y).\end{equation}
Applying $\tau\mapsto \tau+1$ to \eqref{A10A00} and then using \eqref{A10A00} again gives 
\begin{equation}\label{theps}\begin{split}
\Theta_{A_2^\vee,(0,0)}(\epsilon^2 q^{\frac{1}{6}},y)&=2\epsilon \, \Theta_{A_2,(1,0)}(q^{\frac{1}{2}},y)+\Theta_{A_2,(0,0)}(q^{\frac{1}{2}},z)\\
&=
(1-\epsilon)\Theta_{A_2,(0,0)}(q^{\frac{1}{2}},y)+\epsilon \, \Theta_{A^\vee_2,(0,0)}(q^{\frac{1}{6}},y).
\end{split}
\end{equation}
Similarly applying $\tau\mapsto \tau+1$ to \eqref{A01A00} and then using \eqref{theps} gives
\begin{align*} \Theta_{A_2^\vee,(0,1)}(\epsilon^2q^{\frac{1}{6}}, y)&=\frac{3}{2}\Theta_{A_2,(0,0)}(q^{\frac{1}{2}},y)-\frac{1}{2}\Theta_{A_2^\vee,(0,0)}(\epsilon^2 q^{\frac{1}{6}},y)\\
&=\Big(1+\frac{\epsilon}{2}\Big)\Theta_{A_2,(0,0)}(q^{\frac{1}{2}},y)-\frac{\epsilon}{2}\Theta_{A_2^\vee,(0,0)}(q^{\frac{1}{6}},y).
\end{align*}
The other two formulae of \eqref{ident} follow by applying $\tau\mapsto \tau+1$ again. 

The rest of the lemma follows from \eqref{transtheta00} and \eqref{ident}. E.g.~the first two equations of \eqref{ident} imply $(Z(q^{\frac{1}{6}},y)-1) W(q^{\frac{1}{2}},y) = Z(q^{\frac{1}{6}},y)+2$.
\end{proof}

\begin{proposition}
The formula for $\sfZ_{S,H,3,c_1}(q,y)$ from Corollary \ref{corconj1} and Conjecture \ref{conj2} satisfies Conjecture \ref{conj3}.
\end{proposition}
\begin{proof}
\textbf{Step 1:} For $\tau \mapsto \tau+1$, $z \mapsto z$, we use the following transformations
\begin{align*}
\eta(q) &\mapsto e^{\frac{\pi i}{12}} \eta(q), \\
\phi_{-2,1}(q^3,y^3)^{\frac{1}{2}} \Delta(q^3)^{\frac{1}{2}} &\mapsto - \phi_{-2,1}(q^3,y^3)^{\frac{1}{2}} \Delta(q^3)^{\frac{1}{2}}, \\
\phi_{-2,1}(q^{\frac{1}{3}},y)^{\frac{1}{2}} \Delta(q^{\frac{1}{3}})^{\frac{1}{2}} &\mapsto \phi_{-2,1}(\epsilon q^{\frac{1}{3}},y)^{\frac{1}{2}} \Delta(\epsilon q^{\frac{1}{3}})^{\frac{1}{2}}, \\
\phi_{-2,1}(\epsilon q^{\frac{1}{3}},y)^{\frac{1}{2}} \Delta(\epsilon q^{\frac{1}{3}})^{\frac{1}{2}} &\mapsto \phi_{-2,1}(\epsilon^2 q^{\frac{1}{3}},y)^{\frac{1}{2}} \Delta(\epsilon^2 q^{\frac{1}{3}})^{\frac{1}{2}}, \\
\phi_{-2,1}(\epsilon^2 q^{\frac{1}{3}},y)^{\frac{1}{2}} \Delta(\epsilon^2 q^{\frac{1}{3}})^{\frac{1}{2}} &\mapsto - \phi_{-2,1}(q^{\frac{1}{3}},y)^{\frac{1}{2}} \Delta(q^{\frac{1}{3}})^{\frac{1}{2}}, \\
\Theta_{A_2,(0,0)}(q^{\frac{1}{2}},y) &\mapsto \Theta_{A_2,(0,0)}(q^{\frac{1}{2}},y), \\
\Theta_{A_2,(1,0)}(q^{\frac{1}{2}},y) &\mapsto \epsilon \, \Theta_{A_2,(1,0)}(q^{\frac{1}{2}},y), \\
\Theta_{A_2^\vee,v}(q^{\frac{1}{6}},y) &\mapsto \Theta_{A_2^\vee,v}(\epsilon^2 q^{\frac{1}{6}},y), \\
\Theta_{A_2^\vee,v}(\epsilon^2 q^{\frac{1}{6}},y) &\mapsto \Theta_{A_2^\vee,v}(\epsilon q^{\frac{1}{6}},y), \\
\Theta_{A_2^\vee,v}(\epsilon q^{\frac{1}{6}},y) &\mapsto \Theta_{A_2^\vee,v}(q^{\frac{1}{6}},y),
\end{align*}
for both $v = (0,0), (0,1)$. 

We deduce that the contribution of the instanton branch gets mapped to itself up to a factor $(-1)^{\chi} e^{\frac{\pi i}{4} K^2} \epsilon^{-c_1^2}$. The same holds for the contribution of the monopole branch, where we use the following identity
$$
\epsilon^{-K^2} \epsilon^{ab} \epsilon^{(K-a)(K-b)} \delta_{c_1+a,b} = \epsilon^{2ab -a^2-b^2} \delta_{c_1+a,b} = \epsilon^{2(b-a)^2} \delta_{c_1+a,b} = \epsilon^{-c_1^2}  \delta_{c_1+a,b},
$$
which holds for all Seiberg-Witten basic classes $a,b$ by \eqref{SWeqn}. \\

\noindent \textbf{Step 2:} For $\tau \mapsto -1/\tau$, $z \mapsto z/\tau$, we use the following transformations (Lemma's \ref{Dedekind} and \ref{thetalemma})
\begin{align*}
\eta(q) &\mapsto \Big( \frac{\tau}{i} \Big)^{\frac{1}{2}} \eta(q), \\
\phi_{-2,1}(q^3,y^3)^{\frac{1}{2}} \Delta(q^3)^{\frac{1}{2}} &\mapsto - \Big(\frac{\tau}{3}\Big)^5 e^{\frac{3 \pi i z^2}{\tau}} \phi_{-2,1}(q^{\frac{1}{3}},y)^{\frac{1}{2}} \Delta(q^{\frac{1}{3}})^{\frac{1}{2}},  \\
\phi_{-2,1}(\epsilon q^{\frac{1}{3}},y)^{\frac{1}{2}} \Delta(\epsilon q^{\frac{1}{3}})^{\frac{1}{2}} &\mapsto \tau^5 e^{\frac{3 \pi i z^2}{\tau}}\phi_{-2,1}(\epsilon^2 q^{\frac{1}{3}},y)^{\frac{1}{2}} \Delta(\epsilon^2 q^{\frac{1}{3}})^{\frac{1}{2}}, \\
\phi_{-2,1}(\epsilon^2 q^{\frac{1}{3}},y)^{\frac{1}{2}} \Delta(\epsilon^2 q^{\frac{1}{3}})^{\frac{1}{2}} &\mapsto \tau^5 e^{\frac{3 \pi i z^2}{\tau}}\phi_{-2,1}(\epsilon q^{\frac{1}{3}},y)^{\frac{1}{2}} \Delta(\epsilon q^{\frac{1}{3}})^{\frac{1}{2}}, \\
\Theta_{A_2,(0,0)}(q^{\frac{1}{2}},y) &\mapsto \frac{1}{\sqrt{3}} \Big( \frac{\tau}{i} \Big)  e^{\frac{2 \pi i z^2}{\tau}} \Theta_{A_2^{\vee},(0,0)}(q^{\frac{1}{6}},y).  
\end{align*}
The other required transformations can be derived from these using the identities of Lemma \ref{thetalemma}.
The terms involving $\Theta_{A_2^{\vee},(0,1)}(\epsilon q^{\frac{1}{6}},y)$ and $\Theta_{A_2^{\vee},(0,1)}(\epsilon^2 q^{\frac{1}{6}},y)$ in LHS of \eqref{modref} map to the terms involving $\Theta_{A_2^{\vee},(0,1)}(\epsilon^2 q^{\frac{1}{6}},y)$, $\Theta_{A_2^{\vee},(0,1)}(\epsilon q^{\frac{1}{6}},y)$ in RHS of \eqref{modref}. Here we use (besides the transformations listed above):
\begin{itemize}
\item Proposition \ref{LLlattice} and $\epsilon(2) = \big( \frac{1}{3}\big) = 1$, $\epsilon(1) = \big( \frac{2}{3}\big) = -1$,
\item $\epsilon^3 = 1$, where $\epsilon = e^{\frac{2 \pi i}{3}}$, $\chi = \frac{1}{12}(K^2+e)$, and $\sigma = -8\chi+K^2$,
\item replacing summation variables in $\sum_{a,b} \SW(a) \SW(b) \cdots$ by $K-a, K-b$,
\item $aK = a^2$ and $\SW(K-a) = (-1)^\chi \SW(a)$ for all SW basic classes $a$ \eqref{SWeqn}.
\end{itemize}
The term involving $\Theta_{A_2^{\vee},(0,1)}(q^{\frac{1}{6}},y)$ in LHS of \eqref{modref} map to the term involving $\Theta_{A_2,(1,0)}(q^{\frac{1}{2}},y)$ in RHS of \eqref{modref} (and vice versa). 
\end{proof}

\section{Consequences} \label{conseq}

In this section, we discuss some consequences of our conjectures. In particular, we specialize to the following settings:
\begin{itemize}
\item $S$ satisfies $H_1(S,\Z)=0$, $p_g>0$ and only has SW basic classes $0$, $K \neq 0$,
\item $|K|$ contains a reduced curve with irreducible connected components,
\item $\widetilde{S}$ is the blow-up of $S$ in one point.
\end{itemize}

\subsection{Minimal surfaces of general type}

Let $S$ be a smooth projective surface such that $H_1(S,\Z)=0$, $p_g>0$, and its only SW basic classes are $0$ and $K \neq 0$. Then $\SW(0)=0$ and $\SW(K) = (-1)^{\chi}$. Prominent examples are minimal surfaces of general type satisfying $H_1(S,\Z)=0$ and $p_g>0$. In this case, the conjectural formulae for $\sfZ_{S,H,2,c_1}(q,y)$ and $\sfZ_{S,H,3,c_1}(q,y)$ from the introduction simplify as follows (recall that $i = \sqrt{-1}$ and $\epsilon = e^{\frac{2 \pi i}{3}}$):

{{\scriptsize{
\begin{align*}
\frac{\sfZ_{S,H,2,c_1}^{\inst}(q,y)}{(y^{\frac{1}{2}} - y^{-\frac{1}{2}})^{\chi}} =&\, 2\Bigg( \frac{1}{2\phi_{-2,1}(q^{\frac{1}{2}},y)^{\frac{1}{2}} \Delta(q^\frac{1}{2})^{\frac{1}{2}}} \Bigg)^{\chi} \Bigg\{ \Bigg( \frac{\theta_3(q,y)+\theta_2(q,y)}{2\eta(q)^2}\Bigg)^{-K^2} + (-1)^{c_1^2-\chi} \Bigg(\frac{\theta_3(q,y)-\theta_2(q,y)}{2\eta(q)^2}\Bigg)^{-K^2} \Bigg\} \\
&\!\!\!\!\!\!\!\!\!\!\!\!\!\!\!\!\!\!\!\!\!\!\!\!\!\!\!\!\!\!\!\!\!\!\!\!\!\!\!\!\!\!\!\!\!+2 (-1)^{\chi} i^{-c_1^2} \Bigg( \frac{1}{2\phi_{-2,1}(-q^{\frac{1}{2}},y)^{\frac{1}{2}} \Delta(-q^\frac{1}{2})^{\frac{1}{2}}} \Bigg)^{\chi}  \Bigg\{ \Bigg( \frac{\theta_3(q,y)-i\theta_2(q,y)}{2\eta(q)^2}\Bigg)^{-K^2} + (-1)^{c_1^2-\chi} \Bigg(\frac{\theta_3(q,y)+i\theta_2(q,y)}{2\eta(q)^2}\Bigg)^{-K^2} \Bigg\}, \\
\frac{\sfZ_{S,H,2,c_1}^{\mono}(q,y)}{(y^{\frac{1}{2}} - y^{-\frac{1}{2}})^{\chi}} =& \Bigg( \frac{1}{\phi_{-2,1}(q^2,y^2)^{\frac{1}{2}} \Delta(q^2)^{\frac{1}{2}}} \Bigg)^{\chi} \Bigg\{ (-1)^\chi \delta_{c_1,0} \Bigg( \frac{\theta_3(q,y)}{\eta(q)^2} \Bigg)^{-K^2} + \delta_{c_1,K} \Bigg( \frac{\theta_2(q,y)}{\eta(q)^2} \Bigg)^{-K^2} \Bigg\}, \\
\frac{\sfZ_{S,H,3,c_1}^{\inst}(q,y)}{(y^{\frac{1}{2}} - y^{-\frac{1}{2}})^{\chi}} =&\, 3\Bigg( \frac{1}{3\phi_{-2,1}(q^{\frac{1}{3}},y)^{\frac{1}{2}}  \Delta(q^{\frac{1}{3}})^{\frac{1}{2}} } \Bigg)^{\chi} \Bigg(\frac{\Theta_{A_2^\vee,(0,1)}(q^{\frac{1}{6}},y)}{3\eta(q)^3}  \Bigg)^{-K^2} \\
&\qquad\qquad \times \Bigg\{  Z_+(q^{\frac{1}{6}},y)^{K^2} + Z_-(q^{\frac{1}{6}},y)^{K^2}  + (-1)^\chi (\epsilon^{c_1K} + \epsilon^{-c_1K}) \Bigg\} \\
&+3\epsilon^{2c_1^2} \Bigg( \frac{1}{3\phi_{-2,1}(\epsilon^2 q^{\frac{1}{3}},y)^{\frac{1}{2}}  \Delta(\epsilon^2 q^{\frac{1}{3}})^{\frac{1}{2}} } \Bigg)^{\chi} \Bigg(\frac{\Theta_{A_2^\vee,(0,1)}(\epsilon q^{\frac{1}{6}},y)}{3\eta(q)^3}  \Bigg)^{-K^2} \\
&\qquad\qquad \times \Bigg\{  Z_+(\epsilon q^{\frac{1}{6}},y)^{K^2} + Z_-(\epsilon q^{\frac{1}{6}},y)^{K^2}  + (-1)^\chi (\epsilon^{c_1K} + \epsilon^{-c_1K}) \Bigg\} \\
&+3(-1)^\chi \epsilon^{c_1^2} \Bigg( \frac{1}{3\phi_{-2,1}(\epsilon q^{\frac{1}{3}},y)^{\frac{1}{2}}  \Delta(\epsilon q^{\frac{1}{3}})^{\frac{1}{2}} } \Bigg)^{\chi} \Bigg(\frac{\Theta_{A_2^\vee,(0,1)}(\epsilon^2 q^{\frac{1}{6}},y)}{3\eta(q)^3}  \Bigg)^{-K^2} \\
&\qquad\qquad \times \Bigg\{  Z_+(\epsilon^2 q^{\frac{1}{6}},y)^{K^2} + Z_-(\epsilon^2 q^{\frac{1}{6}},y)^{K^2}  + (-1)^\chi (\epsilon^{c_1K} + \epsilon^{-c_1K}) \Bigg\}, \\
\frac{\sfZ_{S,H,3,c_1}^{\mono}(q,y)}{(y^{\frac{1}{2}} - y^{-\frac{1}{2}})^{\chi}} =& \Bigg( \frac{1}{\phi_{-2,1}(q^3,y^3)^{\frac{1}{2}}  \Delta(q^3)^{\frac{1}{2}} } \Bigg)^{\chi} \Bigg(\frac{\Theta_{A_2,(1,0)}(q^{\frac{1}{2}},y)}{\eta(q)^3}  \Bigg)^{-K^2}  \\
&\qquad\qquad \times \Bigg\{ \delta_{c_1,0} \big(W_+(q^{\frac{1}{2}},y)^{K^2} + W_-(q^{\frac{1}{2}},y)^{K^2}\big)  + (-1)^\chi \big( \delta_{c_1,K} + \delta_{c_1,-K} \big) \Bigg\}.
\end{align*}
}}

\subsection{Disconnected canonical divisor} 

Let $S$ be a smooth projective surface such that $b_1=0$, $p_g>0$, and $|K|$ contains a reduced curve with irreducible connected components $C_1, \ldots, C_m$. E.g.~elliptic surfaces over $\PP^1$ of type $E(n)$ with $n \geq 3$, which have $12n$ rational 1-nodal fibres, a section, and no further singular fibres. Then $|K| = |(n-2) F|$, where $F$ is the fibre class. 

For any $I \subset M:=\{1,\ldots, m\}$, define $C_I := \sum_{i \in I} C_i$ and write $I \sim J$ whenever $C_I$ and $C_J$ are linearly equivalent. In \cite[Lem.~6.14]{GK1} we prove that the Seiberg-Witten basic classes are $\{C_I\}_{I \subset M}$ and
\begin{equation} \label{SWdisconn}
\SW(C_I) = |[I]| \prod_{i \in I} (-1)^{h^0(N_{C_i/S})},
\end{equation}
where $|[I]|$ denotes the number of elements of equivalence class $[I]$ and $N_{C_i / S}$ denotes the normal bundle of $C_i \subset S$. 

Suppose $H,c_1$ are chosen such that there are no rank 3 strictly Gieseker $H$-semistable sheaves on $S$ with first Chern class $c_1$. Let
$$
\sfZ_{S,H,3,c_1}^{\inst}(x,y) := \sum_{c_2} \overline{\chi}_{-y}^{\vir}(M_S^H(3,c_1,c_2)) \, x^{\vd},
$$
where $\vd$ is given by \eqref{vd}. Then Conjecture \ref{conj1} applied to $S$ gives
{\scriptsize{
\begin{align*}
\sfZ_{S,H,3,c_1}^{\inst}(x,y) =&\, 3\Bigg( \frac{1}{3 \prod_{n=1}^{\infty}(1-x^{2n})^{10}(1-x^{2n}y)(1-x^{2n}y^{-1})} \Bigg)^{\chi} \Bigg(\frac{\Theta_{A_2^\vee,(0,1)}(x,y)}{3\overline{\eta}(x^6)^3}  \Bigg)^{-K^2} \\
&\qquad\qquad \times \prod_{j=1}^{m} \Bigg( Z_+(x,y)^{C_j^2} + Z_-(x,y)^{C_j^2} + (-1)^{h^0(N_{C_j/S})}(\epsilon^{c_1C_j} + \epsilon^{-c_1C_j}) \Bigg) \\
&+3\epsilon^{2c_1^2+2\chi} \Bigg( \frac{1}{3 \prod_{n=1}^{\infty}(1-\epsilon^{2n} x^{2n})^{10}(1- \epsilon^{2n} x^{2n}y)(1-\epsilon^{2n} x^{2n}y^{-1})} \Bigg)^{\chi} \Bigg(\frac{\Theta_{A_2^\vee,(0,1)}(\epsilon x,y)}{3\overline{\eta}(x^6)^3}  \Bigg)^{-K^2} \\
&\qquad\qquad \times \prod_{j=1}^{m} \Bigg( Z_+(\epsilon x,y)^{C_j^2} + Z_-(\epsilon x,y)^{C_j^2} + (-1)^{h^0(N_{C_j/S})}(\epsilon^{c_1C_j} + \epsilon^{-c_1C_j}) \Bigg) \\
&+3\epsilon^{c_1^2+\chi} \Bigg( \frac{1}{3 \prod_{n=1}^{\infty}(1-\epsilon^{n} x^{2n})^{10}(1- \epsilon^{n} x^{2n}y)(1-\epsilon^{n} x^{2n}y^{-1})} \Bigg)^{\chi} \Bigg(\frac{\Theta_{A_2^\vee,(0,1)}(\epsilon^2 x,y)}{3\overline{\eta}(x^6)^3}  \Bigg)^{-K^2} \\
&\qquad\qquad \times \prod_{j=1}^{m} \Bigg( Z_+(\epsilon^2 x,y)^{C_j^2} + Z_-(\epsilon^2 x,y)^{C_j^2} + (-1)^{h^0(N_{C_j/S})}(\epsilon^{c_1C_j} + \epsilon^{-c_1C_j}) \Bigg).
\end{align*}
}}
The rank 2 analog of this formula is given in \cite[Prop.~6.11]{GK1}.

\begin{proof}[Proof of formula]
As in the proof of Corollary \ref{corconj1}, we denote the formula of Conjecture \ref{conj1} by $\psi_{S,c_1}(x,y) = \sum_{n \geq 0} \psi_n(y) \, x^n$. Then
$$
\sfZ_{S,H,3,c_1}^{\inst}(x,y) = \sum_{n \equiv -2c_1^2 - 8\chi \mod 3} \psi_n(y) \, x^{n} = \sum_{k=0}^{2} \frac{1}{3} \epsilon^{k(2c_1^2+ 8\chi)} \psi_{S,c_1}(\epsilon^k x,y).
$$
Therefore, it suffices to calculate $\psi_{S,c_1}(x,y)$, which equals 
\begin{align*}
&9\Bigg( \frac{1}{3 \prod_{n=1}^{\infty}(1-x^{2n})^{10}(1-x^{2n}y)(1-x^{2n}y^{-1})} \Bigg)^{\chi} \Bigg(\frac{\Theta_{A_2^\vee,(0,1)}(x,y)}{3\overline{\eta}(x^6)^3}  \Bigg)^{-K^2} \\
&\times \sum_{a,b} \SW(a) \SW(b) \epsilon^{(a-b) c_1} Z_+(x,y)^{ab} Z_-(x,y)^{(K-a)(K-b)}.
\end{align*}
Using \eqref{SWdisconn}, the second line becomes
$$
\sum_{I,J} (-1)^{h^0(N_{C_I / S}) + h^0(N_{C_J / S})} \epsilon^{(C_I - C_J)c_1} Z_+(x,y)^{C_I C_J} Z_-(x,y)^{(K-C_I)(K-C_J)},
$$
where the sum runs over all pairs $I,J \subset M$. Writing $I_1 := I \cap J$, $I_2 = I \setminus J$, $I_3 = J \setminus I$, $I_4 = M \setminus (I \cup J)$ and using $K = C_M$, this can be rewritten as
$$
\sum_{I_1 \sqcup I_2 \sqcup I_3 \sqcup I_4 = M} (-1)^{h^0(N_{C_{I_2} / S}) + h^0(N_{C_{I_3} / S})} \epsilon^{(C_{I_2} - C_{I_3})c_1} Z_+(x,y)^{C_{I_1}^2} Z_-(x,y)^{C_{I_4}^2},
$$
where $\sqcup$ stands for disjoint union, from which the result follows.
\end{proof}

\subsection{Blow-ups}

Let $\pi : \widetilde{S} \rightarrow S$ be the blow-up in a point of a smooth projective surface $S$ satisfying $b_1=0$ and $p_g>0$. Let $H,c_1$ be chosen such that there are no rank 3 strictly Gieseker $H$-semistable sheaves on $S$ with first Chern class $c_1$. Furthermore, let
$$
\widetilde{c}_1 = \pi^* c_1 - \ell E,
$$
where $E$ denotes the exceptional divisor and $\ell=0,1,2$. Suppose $\widetilde{H}$ is a polarization on $\widetilde{S}$ such that there are no rank 3 strictly Gieseker $\widetilde{H}$-semistable sheaves on $\widetilde{S}$ with first Chern class $\widetilde{c}_1$. As in the previous section, we consider $\sfZ_{S,H,3,c_1}^{\inst}(x,y)$, $\sfZ_{\widetilde{S},\widetilde{H},3,\widetilde{c}_1}^{\inst}(x,y)$. Conjecture \ref{conj1} applied to $S$, $\widetilde{S}$ gives
\begin{equation} \label{blowupcases}
\sfZ_{\widetilde{S},\widetilde{H},3,\widetilde{c}_1}^{\inst}(x,y) = \left\{\begin{array}{cc} \frac{\Theta_{A_2,(0,0)}(x^3,y)}{\overline{\eta}(x^6)^3} \, \sfZ_{S,H,3,c_1}^{\inst}(x,y) & \mathrm{if \ } \ell=0 \\ \frac{\Theta_{A_2,(1,0)}(x^3,y)}{\overline{\eta}(x^6)^3} \, \sfZ_{S,H,3,c_1}^{\inst}(x,y) & \mathrm{if \ } \ell=1,2.
\end{array}\right.
\end{equation}
Specializing to $y=1$ gives a blow-up formula for virtual Euler characteristics. Surprisingly, the latter coincides with the blow-up formula for topological Euler characteristics \cite[Prop.~3.1]{Got}. The rank 2 analog of \eqref{blowupcases} is \cite[Prop.~6.9]{GK1}.

\begin{proof}[Proof of formula]
Using the same notation as in the previous section, we have
$$
\sfZ_{\widetilde{S},\widetilde{H},3,\widetilde{c}_1}^{\inst}(x,y) = \sum_{k=0}^{2} \frac{1}{3} \epsilon^{k(2\widetilde{c}_1^2+ 8\chi(\O_{\widetilde{S}}))} \psi_{\widetilde{S},\widetilde{c}_1}(\epsilon^k x,y),
$$
where $\widetilde{c}_1^2 = c_1^2 - \ell^2$ and $\chi(\O_{\widetilde{S}}) = \chi(\O_S)$. We calculate $\psi_{\widetilde{S},\widetilde{c}_1}(x,y)$. The Seiberg-Witten basic classes of $\widetilde{S}$ are $\pi^*a$, $\pi^* a + E$, where $a$ runs over all Seiberg-Witten basic classes of $S$, and \cite[Thm.~7.4.6]{Mor}
\begin{equation} \label{SWblowup}
\SW(\pi^*a) = \SW(\pi^*a+E) = \SW(a).
\end{equation}
Conjecture \ref{conj1}, $\chi(\O_{\widetilde{S}}) = \chi(\O_S)$, $K_{\widetilde{S}} = K_S + E$, and \eqref{SWblowup} together imply
\begin{align*}
\psi_{\widetilde{S},\widetilde{c}_1}(x,y) &= \Bigg( \frac{\Theta_{A_2^\vee,(0,1)}(x,y)}{3 \overline{\eta}(x^6)^3} \Bigg) \Big[ Z_+(x,y)^{-1} + Z_-(x,y)^{-1} + \epsilon^k + \epsilon^{-k} \Big] \psi_{S,c_1}(x,y) \\
&= \Bigg( \frac{\Theta_{A_2^\vee,(0,0)}(x,y) +  (\epsilon^k + \epsilon^{-k})\Theta_{A_2^\vee,(0,1)}(x,y) }{3 \overline{\eta}(x^6)^3} \Bigg) \psi_{S,c_1}(x,y),
\end{align*}
where the second equality uses
\begin{align*}
Z_+(x,y)^{-1} + Z_-(x,y)^{-1} &= Z(x,y) = \frac{\Theta_{A_2^\vee,(0,0)}(x,y) }{\Theta_{A_2^\vee,(0,1)}(x,y)},
\end{align*}
which follows from the definition of $Z, Z_{\pm}$. The rest of the proof follows by splitting up the cases $\ell=0,1,2$ and some rewriting using Lemma \ref{thetalemma}.
\end{proof}

{\tt{gottsche@ictp.it, m.kool1@uu.nl}}
\end{document}